\documentclass{amsart}
\usepackage{amssymb, amsfonts, lacromay}
\usepackage{mathrsfs,color,comment}

\usepackage[all,arc]{xy}
\usepackage{enumerate}
\usepackage{mathrsfs}

\newtheorem{thm}{Theorem}[section]
\newtheorem{cor}[thm]{Corollary}
\newtheorem{prop}[thm]{Proposition}
\newtheorem{lem}[thm]{Lemma}

\theoremstyle{definition}
\newtheorem{defn}[thm]{Definition}

\newtheorem{exmp}[thm]{Example}

\newtheorem{notn}[thm]{Notation}

\newtheorem{rem}[thm]{Remark}

\newtheorem{sch}[thm]{Scholium}

\makeatletter
\let\c@equation\c@thm
\makeatother
\numberwithin{equation}{section}

\makeatletter
\ifx\SK@label\undefined\let\SK@label\label\fi
 \let\your@thm\@thm
 \def\@thm#1#2#3{\gdef\currthmtype{#3}\your@thm{#1}{#2}{#3}}
 \def\mylabel#1{{\let\your@currentlabel\@currentlabel\def\@currentlabel
  {\currthmtype~\your@currentlabel}
 \SK@label{#1@}}\label{#1}}
 \def\myref#1{\ref{#1@}}
\makeatother

\newcommand{\bp}{x_0}
\newcommand{\Ch}{\sC\!at(\tilde G,}

\begin{document}

\title{Categorical models for equivariant classifying spaces}

\author{B.J. Guillou}
\address{Department of Mathematics, The University of Kentucky, Lexington, KY, 40506}
\email{bertguillou@uky.edu}
\author{J.P. May}
\address{Department of Mathematics, The University of Chicago, Chicago, IL 60637}
\email{may@math.uchicago.edu}
\author{M. Merling}
\address{Department of Mathematics, The University of Chicago, Chicago, IL 60637}
\email{mmerling@math.uchicago.edu}


 
\begin{abstract}
Starting categorically, we give simple and
precise models for classifying spaces of equivariant principal bundles.
We need these models for work in progress in
equivariant infinite loop space theory and equivariant
algebraic $K$-theory, but the models are of independent
interest in equivariant bundle theory and especially
equivariant covering space theory.
\end{abstract}

\maketitle

\vspace{-8mm}

\tableofcontents

\section*{Introduction}
Let $\PI$ and $G$ be topological groups and let $G$ act on $\PI$, so
that we have a semi-direct product $\GA = \PI \rtimes G$ and a split
extension
\begin{equation}\label{split}  
\xymatrix@1{ 1 \ar[r] & \PI \ar[r]^-{\subset} & \GA \ar[r]^-{q} & G \ar[r] &  1.\\} 
\end{equation}
The underlying space of $\GA$ is $\PI \times G$, and the product is given by
\[  (\si,g)(\ta,h) = (\si (g\cdot\ta),gh).\]
There is a general theory of $(G,\PI_G)$-bundles \cite{tD, La, LM, EHCT} corresponding 
to such extensions. Here $\PI_G$ denotes $\PI$ together with its given action of $G$.
We shall only be interested in principal $(G,\PI_G)$-bundles $p\colon E\rtarr B$.

\begin{defn}\mylabel{first1} Let $p\colon E\rtarr B$ be a principal $\PI$-bundle where $B$ is a
$G$-space.  Then $p$ is a principal $(G,\PI_G)$-bundle if the (free) action of 
$\PI$ on $E$ extends to an action of $\GA$ and $p$ is a $\GA$-map, where $\GA$ 
acts on $B$ through the quotient map $\GA\rtarr G$. 
\end{defn}

The more general theory of 
$(\PI;\GA)$-bundles applicable to non-split extensions $\GA$ is included in 
\cite{LM, May, EHCT}.  The theory is especially familiar when $G$ acts trivially 
on $\PI$, so that $\GA = G\times \PI$.  With $\PI=O(n)$ 
or $U(n)$, the trivial action case gives classical equivariant bundle theory 
and equivariant topological $K$-theory.

\begin{defn}\mylabel{first2}
A principal $(G,\PI_G)$-bundle $p\colon E\rtarr B$ is universal if for all $G$-spaces $X$ of
the homotopy types of $G$-CW complexes, pullback of $p$ along $G$-maps $f\colon X\rtarr B$ induces a natural bijection from the
set of homotopy classes of $G$-maps $X\rtarr B$ to the set of equivalence classes of $(G,\PI_G)$-bundles over $X$.
\end{defn}

For applications in equivariant
infinite loop space theory and equivariant algebraic $K$-theory, we need 
to understand classifying $G$-spaces for $(G,\PI_G)$-bundles as 
classifying spaces of categories.  Nonequivariantly, it was already emphasized 
in Segal's classical paper \cite[\S3]{Segal} that the universal principal $\PI$-bundle of a topological
group $\PI$ can be constructed on the level of topological categories, and
the intuition is that we are giving the equivariant generalization of 
his classical construction.  

One motivation is to give new constructions of $E_{\infty}$ operads of $G$-categories and
$G$-spaces.  This much only requires trivial actions of $G$ on $\PI$.  
By definition, the $j$th-space of an $E_{\infty}$ operad of $G$-spaces is a universal
principal $(G,\SI_j)$-bundle.  Having various category level models for such classifying
spaces allows us to construct examples of $E_{\infty}$ $G$-spaces from $E_{\infty}$ categories, 
and these feed into equivariant infinite loop space machines to construct interesting $G$-spectra \cite{GM, MMO}. 

The examples relevant to the equivariant algebraic $K$-theory of $G$-rings, namely rings with $G$-action by automorphisms, 
require more general split extensions. If $R$ is a $G$-ring, then $G$ acts entrywise on $GL(n, R)$.  The classifying spaces 
of $(G, GL(n, R)_G)$-bundles are central to the definition of the genuine equivariant algebraic $K$-theory spectrum $\bK_G(R)$ 
of $R$ \cite{GM, Merling}. Our treatment of the fixed point spaces of the classifying spaces of equivariant bundles is crucial
to determining the fixed point spectra of the $\bK_G(R)$. The paradigmatic example is a finite Galois extension $E/F$ with 
Galois group $G$.  As explained in \cite[\S\S 4.5, 8.2]{GM}, it is an immediate 
application of examples in this paper, which demonstrate the relevance of Hilbert's theorem 90, that the fixed point spectrum
$\bK_G(E)^H$ is the classical nonequivariant $K$-theory spectrum of the fixed field $E^H$.  
The use of genuine $G$-spectra in algebraic $K$-theory is new and is explored in \cite{Merling}. 

The results we need are close to
those of \cite{La, LM, May} and those stated by Murayama and 
Shimakawa \cite{MS}\footnote{But see \myref{scholium}.},
but we require a more precise and rigorous categorical and topological 
understanding than the literature affords. This is intended as a service paper 
that displays the relevant constructions in their fullblown simplicity.  

We start with the topologized equivariant version of the
elementary theory of chaotic categories in \S\ref{Sec1}.
We analyze a general construction that specializes 
to give our classifying $G$-spaces in \S\ref{Sec2}.  
We show how it gives universal equivariant bundles in \S\ref{Sec3}.
Our explicit description of the classifying spaces of $(G,\PI_G)$-bundles 
as classifying spaces of categories allows us to compute their fixed point 
spaces categorically in \S\ref{Sec4}.  
This gives precise information already on the category level, before 
passage to classifying spaces, and that is essential to our applications. 

The main results of the paper  are summarized in the following two theorems: the first gives 
a categorical model for equivariant universal bundles and their classifying spaces, and the 
second gives a description of the fixed points of the classifying spaces of equivariant bundles.
Details of the first are in Theorems \ref{universal} and \ref{Main} and details of the second
are in \ref{FFFxxx}, \ref{FixFix}, and \ref{FixToo}. We need some preliminary definitions and
notations to state these results.

Let $G$ be discrete and let $\tilde{G}$ denote the unique contractible groupoid 
with object set $G$. It is a (right) $G$-category, meaning that $G$ acts on both 
objects and  morphisms.  We agree to identify the topological group $\PI$ 
with the topological groupoid with a single object and with morphism space $\PI$. 
Then the  action of $G$ on $\PI$ makes it a $G$-groupoid. 

For small topological categories $\sA$ and $\sB$, let $\sC\!at(\sA,\sB)$ denote the category of all continuous 
functors $\sA\rtarr \sB$ and all natural transformations.  
When $\sA$ and $\sB$ are $G$-categories, 
$\sC\!at(\sA,\sB)$ inherits an action of $G$
given by conjugation.
We shall give more details in \S\ref{prelim}. 

We assume that the reader is familiar with the classifying space functor $B$ from categories to spaces, 
or more generally from topological categories to spaces.
It works equally well to construct $G$-spaces from 
topological $G$-categories. It is the composite of the nerve 
functor $N$ from topological categories to simplicial spaces
(e.g. \cite[\S7]{MayClass}) and geometric realization $|-|$ from 
simplicial spaces to spaces (e.g. \cite[\S11]{MayGeo}),
both of which are product-preserving functors. 

\begin{thm}
If $G$ is discrete and $\PI$ is either discrete or a 
compact Lie group, then the  canonical map
$$B  \Ch \tilde{\PI}) \rtarr B \Ch \PI) $$ 
is a universal principal $(G, \PI_G)$-bundle.
\end{thm}

Thus the classifying space of the $G$-category $\Ch \PI)$ is a  $G$-space that classifies $(G,\PI_G)$-bundles. 

Crossed homomorphisms, their automorphism groups, and the non-Abelian cohomology group $H^1(G;\PI_G)$
are defined in Definitions \ref{crossedhom}, \ref{crossedhom2}, and \ref{crossedhom3}. 

\begin{thm}
The fixed point category 
$\Ch \PI)^G$ is the disjoint union of the groups ${Aut}\,\al$,
where $\al$ runs over crossed homomorphisms representing the elements of $H^1(G;\PI_G)$.
Equivalently, $\Ch \PI)^G$ is the disjoint union of the groups
$\PI\cap N_{\GA}\LA$, where $\LA$ runs over the $\PI$-conjugacy classes of subgroups 
$\LA$ of $\GA$ such that 
$\LA\cap \PI = e$. Therefore $B\Ch \PI)^G$ is the disjoint union of the classifying
spaces $B(\PI\cap N_{\GA}\LA)$.

\end{thm}

With more work, our hypotheses on $G$ and $\PI$ could surely be weakened.
We should admit that we are especially interested in discrete groups in our
applications.  When $\PI$ is discrete, we are in effect studying equivariant covering spaces, since $\PI$ is the relevant structural group. We keep the topology 
in place whenever we can do so easily because the results should have broader applications.  
In fact, however, there is an earlier topological analogue of our categorical construction in terms of mapping spaces rather than
mapping categories \cite{May}. It applies in considerably greater topological generality, 
but it does not generally start categorically.
We compare the categorical and topological constructions in \S\ref{Sec5}.

The choices of $\PI$ relevant to equivariant infinite loop space theory and equivariant algebraic $K$-theory, namely symmetric groups and the general linear groups of $G$-rings, have alternative categorical models, which play a key role. These alternative categorical models are given in \S6, which is entirely algebraic,
with all groups discrete.  We call special attention to \S6.2, where we relate crossed homomorphisms to skew group rings and their skew modules.  The algebraic ideas here may not be as well-known as they should be and deserve further study.

The letter $B$ for the classifying space functor from categories to spaces would sometimes be awkward in 
our context, since the classifying space functor will also be used to construct universal bundles rather 
than classifying spaces for bundles, hence we agree to 
write out $|N-|$ rather than $B$ whenever $B$ seems likely to confuse. 

This notation also displays a key technical problem that is sometimes overlooked
in the literature. The functor $|-|$ is a left adjoint and therefore preserves all colimits, such as passage to orbits in the equivariant setting. The functor $N$ is 
a right adjoint and it generally does not preserve colimits or passage to orbits, 
as we illustrate with elementary examples.  This problem is the subject of the paper \cite{Bab} by Babson and Kozlov. For topological categories, there is no discussion 
in the literature. Exceptionally, $N$ does commute with passage to orbits in the key examples that appear in equivariant bundle theory. Clear understanding of passage to orbits is essential to our calculations of fixed point spaces. 

\begin{rem} The functor $\Ch -)$ from $G$-categories
to $G$-categories plays a central role in our work.  Its $G$-fixed category
was introduced by Thomason \cite[(2.1)]{Thom}, who called it the lax limit of the 
action of $G$ on $\sC$ and denoted it by
$\sC\!at_G(\ul{EG},\ul C)$.  The relevance to equivariant bundle theory of 
the equivariant precursor $\Ch \sB)$  was first 
noticed by Shimakawa \cite{MS,Shim}. 
\end{rem}

\subsection*{Acknowledgements} The third author thanks Matthew Morrow and Liang Xiao 
for answers to her questions that pointed out the striking relevance to our
work of $H^1(G;\PI)$ and Serre's general version of Hilbert's Theorem 90.
We are grateful to an anonymous referee for a careful reading and suggestions
for improving the notations and exposition.

\section{Preliminaries on chaotic and translation categories}\label{Sec1}
The definitions we start with are familiar and elementary. However,
to keep track of categorical data and group actions later, we shall be pedantically precise.

\subsection{Preliminaries on topological $G$-categories}\label{prelim}
Let $\sC\!at$ be the category of categories and functors.  We may 
also view it as the $2$-category of categories, with $0$-cells, $1$-cells,
and $2$-cells the categories, functors, and natural transformations.  From
that point of view, $\sC\!at(\sA,\sB)$ is the internal hom category whose objects are 
the functors $\sA\rtarr \sB$ and whose morphisms are the natural transformations 
between them; they enrich $\sC\!at$ over itself.

For a group $G$, a $G$-category $\sA$ is a category with an action of $G$
specified by a homomorphism from $G$ to the automorphism group of $\sA$. 
Regarding $G$ as a groupoid with one object, the action is specified by a 
functor $G\rtarr \sC\!at$. We have the $2$-category $G\sC\!at$ of 
$G$-categories, $G$-functors, and $G$-natural transformations, where the 
latter notions are defined in the evident way: everything must be equivariant.

We may view $G\sC\!at$ as the underlying $2$-category of a category enriched 
over $G\sC\!at$. The $0$-cells are still $G$-categories, but now we have the 
$G$-category $\sC\!at(\sA,\sB)$ as the internal hom between them.  Its 
underlying category is $\sC\!at(\sA,\sB)$, and $G$ acts by conjugation on 
functors and natural transformations.  Thus, for $F\colon \sA\rtarr \sB$, 
$g\in G$, and $A$ either an object or a morphism of $\sA$, 
$(gF)(A) = gF(g^{-1}A)$.  Similarly, for a natural transformation 
$\et\colon E\rtarr F$ and an object $A$ of 
$\sA$, 
$$(g\et)_A = g\et_{g^{-1}A}\colon g E(g^{-1}A)\rtarr gF(g^{-1}A).$$
The category $G\sC\!at(\sA,\sB)$ is the same as the $G$-fixed category 
$\sC\!at(\sA,\sB)^G$, and we sometimes vary the choice of notation.

We can topologize the definitions so far, starting with the $2$-category of categories 
internal to the category $\sU$ of (compactly generated) spaces, together with continuous 
functors and continuous natural transformations.  Recall that a category $\sA$ internal to a cartesian monoidal category $\sV$ has object and morphism objects in $\sV$ and structure maps Source, Target, Identity and Composition in $\sV$.
These maps are denoted $S$, $T$, $I$, and $C$, and the usual category axioms must hold.
When $\sV = \sU$, we refer to internal categories as topological categories; we refer to them as topological $G$-categories when $\sV=G\sU$. These are more general than (small) topologically enriched categories, which have discrete sets of objects. We can now 
allow $G$ to be a topological group in the equivariant picture.  We continue to use the notations already given in the more general topological situation. 

\subsection{Chaotic topological $G$-categories}
\begin{defn}\mylabel{chacat} A small category $\sC$ is {\em chaotic} if there is exactly one morphism from $b$ to $a$ 
for each pair of objects $a$ and $b$. The unique morphism from $a$ to $b$ must then be inverse to the
unique morphism from $b$ to $a$. Thus $\sC$ is a groupoid, and its classifying space is contractible since every object 
is initial and terminal; in fact, it is the unique contractible
groupoid with the given object set. A topological category $\sC$ is 
chaotic if its underlying category is chaotic.  Its classifying space is again contractible (see \myref{simpcon}), but there are other topological
groupoids with the given object space and contractible classifying spaces. Similarly, a topological $G$-category is chaotic if its underlying category is chaotic.  It is then contractible but not usually 
$G$-contractible.
\end{defn}

The senior author remembers hearing the name ``chaotic'' long ago,
but we do not know its source.  The idea is that everything is the 
same as everything else, which does seem rather chaotic.\footnote{Some
category theorists suggest the name ``indiscrete category'', by formal analogy
with indiscrete spaces in topology. The key difference is that indiscrete spaces 
are of no interest, whereas we hope to convince the reader that chaotic categories
are of considerable interest.}

\begin{lem}\mylabel{chaochao} If $\sA$ is any category and $\sB$ is a chaotic category, then the category 
$\sC\!at(\sA,\sB)$ is again chaotic.
\end{lem}
\begin{proof} The unique natural map $\ze\colon E\rtarr F$ between 
functors $E,F\colon \sA\rtarr \sB$ is given on an object $A$ of $\sA$ by the
unique map $\ze_A\colon E(A)\rtarr F(A)$ in $\sB$.
\end{proof}

\begin{lem}\mylabel{chaochao2} If $\sA$ is any topological $G$-category and $\sB$ is a chaotic 
topological $G$-category, then the topological $G$-category $\sC\!at(\sA,\sB)$ 
and its $G$-fixed category $G\sC\!at(\sA,\sB)$ are again chaotic.
\end{lem}
\begin{proof} Since $\sC\!at(\sA,\sB)$ is just the category
$\sC\!at(\sA,\sB)$ with its conjugation action by $G$, \myref{chaochao} implies
the conclusion for $\sC\!at(\sA,\sB)$. The conclusion is inherited
by $G\sC\!at(\sA,\sB)=\sC\!at(\sA,\sB)^G$ since the unique natural transformation between 
$G$-functors $E$ and $F$ is necessarily a $G$-natural transformation. 
\end{proof}

\begin{defn}  The chaotic topological category $\tilde{X}$ generated by a 
space $X$ is the topological category with object space $X$ and morphism space 
$X\times X$.  The source, target, identity, and composition maps are defined by
$$ S = \pi_2\colon X \times X \rtarr X , \ \ T = \pi_1\colon X \times X \rtarr X , \ \ 
I = \DE\colon X \rtarr X \times X , \ \ \text{and}$$
$$ C = \id\times \epz\times \id \colon (X \times X )\times_X (X \times X ) 
\iso X \times X \times X \rtarr X \times X,$$
where $\epz\colon X\rtarr \ast$ is the trivial map.
On elements, $S(y,x) = x$, $T(y,x) = y$, $I(x) = (x,x)$, and $C(z,y,x) = (z,x)$.
Forgetting the topology, the element $(y,x)$ is the unique morphism 
$x\rtarr y$.  Reversing the order of source and target in the notation this way, 
so that $(z,y)\com (y,x) = (z,x)$, will turn out to be helpful later.

A map $f\colon X\rtarr Y$ induces the functor $\tilde f\colon \tilde X\rtarr \tilde Y$
given by $f$ on objects and $f\times f$ on morphisms.  When $X$ is a (left or right) $G$-space, 
we give $\tilde{X}$ the action specified by the given action on the object space $X$ and the 
diagonal action on the morphism space $X\times X$; $\tilde{X}$ is then a
chaotic topological $G$-category. Sending $X$ to $\tilde X$ specifies a functor 
from the category $G\sU$ of $G$-spaces to the category $G{\sG}\!pd$ of topological $G$-groupoids (a full 
subcategory of $G\sC\!at$).
\end{defn} 

\subsection{The adjunction between $G$-spaces and topological $G$-groupoids}

Sending a category to its set of objects restricts to an object functor $\sO\!b\colon G{\sG}\!pd\rtarr G\sU$.

\begin{lem}\mylabel{chaosadj} The chaotic category functor
is right adjoint to the object functor, so that 
\[ G\sC\!at(\sC,\tilde{X})= G\sG\!pd(\sC,\tilde{X})\iso G\sU(\sO\!b\sC,X)\] 
for a topological $G$-groupoid $\sC$ with object space $\sO\!b\sC$ and a topological $G$-space $X$. 
If $\sC$ is chaotic with object $G$-space $X$, then the unit of the adjunction is an isomorphism of topological $G$-groupoids 
$\et \colon \sC \rtarr \tilde X$.
\end{lem}

\begin{proof} Let $\sM\!or\sC$ be the morphism $G$-space of $\sC$. 
The functor $\sC\rtarr \tilde{X}$ determined by a continuous $G$-map 
$f\colon \sO\!b\sC\rtarr X$ is $f$ on object $G$-spaces and the composite 
$$ \xymatrix@1{ 
\sM\!or\sC\ar[r]^-{(T,S)} & \sO\!b\sC\times \sO\!b\sC\ar[r]^-{f\times f} & X\times X\\} $$
on morphism $G$-spaces. The last statement rephrases the meaning of chaotic. 
\end{proof}
  
\subsection{Translation categories and chaotic categories}

We use another simple definition to
relate chaotic categories to other familiar categories.
Let $G$ be a topological group and $Y$ be a left $G$-space.
Generalizing how we think of $G$ as a one object category,
we can think of $Y$ together with its action by $G$ as the 
functor $Y\colon G\rtarr \sU$ that sends the single object $\ast$ to $Y$ 
and is given on morphism spaces by the
map $G\rtarr \sU(Y,Y)$ adjoint to the action map $G\times Y\rtarr Y$.

\begin{defn}\mylabel{elt} Let $Y$ be a left $G$-space. Define the
translation category $T(G,Y)$ as follows. The object space is 
$Y$ and the morphism space is $G\times Y$. We think of $(g,y)$ as a
morphism $g\colon y\rtarr gy$. The map $I\colon Y\rtarr G \times Y$ sends $y$ to 
$(e,y)$. The maps $S$ and $T$ send $(g,y)$ to $y$ and $gy$, respectively.
The domain of composition, $(G \times Y)\times_{Y} (G \times Y)$, can be
identified with $(G\times G)\times Y$, and composition sends $(h,g,y)$
to $(hg,y)$. The construction is functorial in $Y$, for fixed $G$, and 
in the pair $(G,Y)$ in general.  If $Y$ has a right action by $G$ that
commutes with the left action, then $T(Y,G)$ is a right $G$-category
via the given right action on the object space $Y$ and on the second coordinate
of the morphism space $G\times Y$. 
\end{defn}

\begin{rem} The definition makes sense when $G$ is only a monoid, not 
necessarily a group.  When $Y$ is a point, $T(Y,G)$ is $G$ regarded
as a one object category. When $G$ is a group, $T(Y,G)$ is the 
standard groupoid associated to a $G$-space, but it is not generally 
chaotic.   
\end{rem}

\begin{prop}\mylabel{silly} For left $G$-spaces $Y$, there is a natural comparison
functor $\mu\colon T(G,Y) \rtarr \tilde{Y}$.  If $Y$ has a right action that
commutes with its left action, then $\mu$ is a map of right $G$-categories. The
functor $\mu\colon T(G,G) \rtarr \tilde{G}$ is an isomorphism of right $G$-categories.
\end{prop}  
\begin{proof}
Define $\mu$ to be the identity map on object spaces and the map
that sends $(g,y)$ to $(gy,y)$ on morphism spaces. Since $\tilde{Y}$ is
chaotic, this is the unique functor that is the identity on objects, and
it is easy to check equivariance when $Y$ has a right $G$-action.  When $Y=G$ with 
left action and right action
given by its product, $\mu$ is an isomorphism with $\mu^{-1}(h,g) = (hg^{-1},g)$ 
on morphism spaces.
\end{proof}

In view of the differing group actions on the morphism spaces $G\times G$, action 
on the right coordinate in $T(G,G)$ and diagonal action in $\tilde{G}$, the 
isomorphism between $T(G,G)$ and $\tilde{G}$ must not be viewed as a tautology.

\begin{rem} When we return to the split extension (\ref{split}), the 
group $\PI$ there will play a role close to that of the group denoted $G$ 
in \myref{elt} and \myref{silly}. When $G=e$, we would then specialize 
to $Y=\PI$ with its natural left $\PI$ action and see the usual universal
principal $\PI$-bundle.  When $G\neq e$, the relevant specialization is a little
less obvious; see \myref{silly2}, which is a follow up of \myref{silly}.  
\end{rem}

\section{The category $\sC\!at(\tilde{X},\PI)$}\label{Sec2}
We let $X$ be a space and $\PI$ be a topological group in this section. 
We regard $\PI$ as a category with one object without change of notation; it should be clear from the 
context when we mean the group $\PI$ and when we mean the category $\PI$.
From now on, functors and natural transformations are to be continuous 
(in the topological sense), even when we neglect to say so.
We are especially interested in the functor categories 
$\sC\!at(\tilde{X},\tilde{\PI})$, which are chaotic by \myref{chaochao}, 
and in the functor categories $\sC\!at(\tilde{X},\PI)$, which are not. 
The right action of $\PI$ on $\tilde{\PI}$ induces a right action of $\PI$ on
$\sC\!at(\tilde{X},\tilde{\PI})$.

This section and the next give a pedantically explicit description of $\sC\!at(\tilde{X},\PI)$ 
and of the induced map 
$$\sC\!at(\tilde{X},\tilde{\PI})\rtarr \sC\!at(\tilde{X},\PI),$$
showing in particular that it is obtained by passage to orbits over $\PI$.  When $X=G$, this 
elementary analysis will
be at the heart of all our proofs.  We defer adding in the second group $G$ 
that appears in the bundle theory until after we have this description in place 
since a group defined solely in terms of the diagonal on $X$ and the product on 
$\PI$ plays a central role in the description. 

\subsection{An explicit description of $\sC\!at(\tilde{X},\PI)$}\label{explicit}
By the adjunction given in \myref{chaosadj} (with $G=e$),  the object space of the chaotic 
category $\sC\!at(\tilde{X},\tilde{\PI})$ can be identified with the space $\sU(X,\PI)$ of maps 
$X\rtarr \PI$ with its standard (compactly generated) function space topology.  Therefore
$\sC\!at(\tilde{X},\tilde{\PI})$ can be identified with the chaotic category $\widetilde{\sU(X,\PI)}$.

\begin{defn} Define the pointwise product $\ast$ on $\sU(X,\PI)$ by 
\[(\al\ast \be)(x) = \al(x)\be(x) \]
for $\al,\be\colon X\rtarr \PI$. The unit element $\varepsilon$ is given by 
$\varepsilon(x) = e$ and inverses are given by $\al^{-1}(x) = \al(x)^{-1}$.
The topological group $\sU(X,\PI)$ contains $\PI$ as a (closed) subgroup, where 
we regard an element $\si\in\PI$ as the constant map $\si\colon X\rtarr \PI$ at $\si$. 
The inclusion of $\PI$ in $\sU(X,\PI)$ and composition give $\sU(X,\PI)$ its right $\PI$-action.
\end{defn}

\begin{defn} Choose a basepoint $\bp\in X$.  There is a unique representative 
map $\al$ such that $\al(\bp)=e$ in each orbit of $\sU(X,\PI)$ under the right 
action by $\PI$. Let $\sO(X,\PI)\subset \sU(X,\PI)$ denote the subspace of such 
representative maps. It is a subgroup of $\sU(X,\PI)$.  The $\PI$-action 
and the product $\ast$ on $\sU(X,\PI)$ are related by $\al\si = \al\ast\si$
for $\si\in \PI$, and $\ast$ restricts to a homeomorphism
of $\PI$-spaces $\sO(X,\PI)\times \PI\rtarr \sU(X,\PI).$
Write elements of $\sU(X,\PI)$ in the form $\al\si$, where $\al(\bp)=e$. 
Passage to orbits restricts to a homeomorphism  $\sO(X,\PI)\iso \sU(X,\PI)/\PI$.
Observe that the product $\ast$ on $\sU(X,\PI)$ induces a left action of 
$\sU(X,\PI)$ on $\sO(X,\PI)$ by sending $(\be,\al)$ to the orbit representative
of $\be\ast \al$. 
\end{defn}

The proofs of the follow three lemmas are simple exercises from the fact that there is a 
unique morphism $(y,x)$ from $x$ to $y$ in $\tilde{X}$; compare \myref{chaochao}.

\begin{lem}\mylabel{Ob(X,P)} A functor $E\colon \tilde X\rtarr \PI$
is given by the trivial map $X\rtarr\ast$ of object spaces and a map 
$E\colon X\times X\rtarr \PI$ of morphism spaces 
such that $E(x,x)=e$ and $E(z,y)E(y,x) = E(z,x)$. Define $\al\in \sO(X,\PI)$
by $\al(x) =E(x,\bp)$.  Then $\al$ determines $E$ by the formula
$$E(y,x) = E(y,\bp)E(\bp,x) = \al(y)\al(x)^{-1}.$$
Writing $E=E_{\al}$, sending $E_{\al}$ to $\al$ specifies a homeomorphism from
the space of functors $\tilde X\rtarr \PI$ to $\sO(X,\PI)$.
\end{lem}

\begin{lem}\mylabel{Mor(X,P)} For $E_{\al},E_{\be}\colon \tilde{X}\rtarr \PI$, 
a natural transformation $\et\colon E_{\al}\rtarr E_{\be}$ is given by a map
$\et\colon X\rtarr \PI$ such that $\et(y)E_{\al}(y,x) = E_{\be}(y,x)\et(x)$
for $x,y\in X$.  If $\si\in \PI$ is defined by $\si = \et(\bp)$, then the pair
$(\be\si,\al)$ determines $\et$ by the formula
$$\et(x) = E_{\be}(x,x_0)\et(\bp) E_{\al}(x,\bp)^{-1} = (\be\si\ast \al^{-1})(x).$$
Writing $\et = \et_{\si}$, sending 
$\eta_{\si}$ to $(\be\si,\al)$ specifies a homeomorphism from the space 
of morphisms of $\sC\!at(\tilde{X},\PI)$ to the space $\sU(X,\PI)\times \sO(X,\PI)$.
\end{lem}

\begin{lem}\mylabel{Com(X,P)} Identify the object and morphism spaces of 
$\sC\!at(\tilde{X},\PI)$ with 
$$\sO(X,\PI) \ \ \text{and}\ \  \sM(X,\PI)\equiv \sU(X,\PI)\times \sO(X,\PI)  $$
via the homeomorphisms of the previous two lemmas.
Then the identity map $I$ sends $\al$ to $(\al e,\al)$ and the source and target maps 
$S$ and $T$ send $(\be\si, \al)$ to $\al$ and to $\be$. The $S=T$ pullback 
\[ \sM(X,\PI)\times_{\sO(X,\PI)}\sM(X,\PI)  \]
can be identified with $\sU(X,\PI)\times\sU(X,\PI) \times \sO(X,\PI)$ via
\[ ((\ga\ta,\be),(\be\si,\al)) \leftrightarrow (\ga\ta,\be\si,\al) \]
and the composition map $C$ sends 
$(\ga\ta,\be\si,\al)$ to $(\ga\ta\si,\al)$. 
\end{lem} 
\begin{proof}  If we compose
$\et_{\ta}\colon E_{\be}\rtarr E_{\ga}$ with 
$\et_{\si}\colon E_{\al} \rtarr E_{\be}$, we obtain
\[ \et_{\ta}\ast \et_{\si} 
= \ga^{-1}\ta \ast \be \ast \be^{-1}\si \ast \al
= \ga^{-1}\ta\si\ast\al, \]
which corresponds to the given description.
\end{proof}

\subsection{Two identifications of $\sC\!at(\tilde X,\PI)$}

We show here that \myref{silly} leads to one identification of 
$\sC\!at(\tilde X,\PI)$, and
the lemmas of the previous section lead to a closely related one.  
These elementary identifications commute passage to orbits with 
the functor $\sC\!at(\tilde X,-)$, and that will be crucial to understanding
$B\sC\!at(\tilde G,\PI)$ as an equivariant classifying space.

\begin{notn}\mylabel{psqs} The category $\PI$ is isomorphic to the orbit 
category $\tilde{\PI}/\PI$. The quotient functor 
$p\colon \tilde{\PI}\rtarr \PI$ 
is the trivial map $\PI\rtarr \ast$ on object spaces and is given on 
morphism spaces by the map 
$\xymatrix@1{p\colon \PI\times \PI \rtarr (\PI\times \PI)/\PI \iso \PI\\}$
specified by $p(\ta,\si) = \ta\si^{-1}$. 
Let $q$ denote the functor 
\[ \sC\!at(\id,p) \colon  \sC\!at(\tilde X,\tilde \PI)\rtarr  
\sC\!at(\tilde X,\PI). \\ \]
We also let $q$ denote the functor between translation categories
\[ T(\sU(X,\PI),\sU(X,\PI)) \rtarr T(\sU(X,\PI),\sO(X,\PI)) \]
that is induced by the quotient map $p\colon \sU(X,\PI)\rtarr \sU(X,\PI)/\PI\iso \sO(X,\PI)$.
\end{notn}

\begin{thm}\mylabel{notformal} 
There is a commutative diagram of topological categories in which $\mu$, $\nu$, and
$\xi$ are isomorphisms.
\[ \xymatrix{
T(\sU(X,\PI),\sU(X,\PI)) \ar[rr]^-{\mu} \ar[d]_q 
& &\sC\!at(\tilde X,\tilde \PI) 
\ar[dl]_{p} \ar[d]^{q} \\
T(\sU(X,\PI),\sO(X,\PI)) \ar[r]_-{\nu} 
& \sC\!at(\tilde X,\tilde \PI)/\PI
\ar[r]_{\xi} & \sC\!at(\tilde X,\PI) \\ } \] 
\end{thm}
\begin{proof} The map $p$ is the quotient map given by passage to orbits 
over $\PI$. Since $q$ on the right is a $\PI$-map with $\PI$ acting trivially on $\sC\!at(\tilde X,\PI)$, 
$q$ factors through a map $\xi$ that makes the triangle commute.  Since 
$\sC\!at(\tilde X,\tilde \PI)$ is the chaotic category whose object space
is the topological group $\sU(X,\PI)$, \myref{silly} gives the isomorphism
$\mu$. Since $q$ on the left is obtained by passage to orbits from the relevant
action of $\PI$, it is clear that $\mu$ induces an isomorphism $\nu$ making
the left trapezoid commute.  

All that remains is to prove that $\xi$
is an isomorphism, and that follows from the results of \S\ref{explicit}.
For a functor $E_{\al}\colon \tilde{X}\rtarr \PI$, $\al\colon X\rtarr \PI$
and $\al\times \al\colon X\times X\rtarr\PI\times \PI$ 
define the object and morphism maps of a functor $F\colon \tilde{X}\rtarr \tilde{\PI}$. The functoriality
properties of $E_{\al}$ show that $p\com F = E_{\al}$, so that $q$ is surjective on objects. If we also have $p\com F' = E_{\al}$, then a quick check shows that $F(x)^{-1}F'(x) = F(y)^{-1}F'(y)$ for all $x,y\in X$.  If the common value is denoted by
$\si$, then $F'(x) = F(x)\si$ for all $x$.  In view of the specification of $p$ and
$q$ in
\myref{psqs}, this implies that $\xi$ is a homeomorphism on object spaces.

Now let $E_{\al},E_{\be}\colon \tilde{X}\rtarr \PI$ be any two functors. 
For any choices of functors $F,F'\colon \tilde{X} \rtarr \tilde{\PI}$ such that 
$q\com F = E_{\al}$ and $q\com F' = E_{\be}$, define $\ze\colon X\rtarr \PI\times \PI$  by $\ze(x) = (F(x),F'(x))$. Then $\ze$ is a map from the object space of $\tilde{X}$ to the morphism space of $\tilde{\PI}$. A quick check shows that $\ze$ is a natural transformation $F\rtarr F'$ such that $\et = q\com \ze$ is a natural transformation 
$E_{\al}\rtarr E_{\be}$ with $\et_{\bp} = F'(\bp)F(\bp)^{-1}$.  Via our enumeration
of the possible choices, this implies that $q$ restricted to the inverse image of the
space of natural transformations $E_{\al}\rtarr E_{\be}$ can be identified with the
quotient map $p\colon \PI\times \PI\rtarr \PI$ of \myref{psqs}.  It follows that 
$\xi$ is a homeomorphism on morphism spaces.
\end{proof}

\subsection{The nerve functor and classifying spaces}\label{Secnerve}

We recall the definition of the nerve functor $N$ in more detail than might be thought
warranted at this late date since, in the presence of the left-right action dichotomy 
of multiple group actions, the original definitions in category theory can cause real 
problems arising from categorical dyslexia. There are two standard conventions in the
literature, and we must choose.  Let $\sC$ be a topological category with 
object space $\sO$ and morphism space $\sM$. Then $N_0\sC =\sO$ and, for $q>0$,
\[ N_q\sC = \sM \times_{\sO} \cdots \times_{\sO} \sM, \] 
with $q$ factors $\sM$.  The pullbacks are over pairs of maps $(S,T)$. 
To avoid dyslexia, we remember that $g\com f$ means first $f$ and then $g$,
and choose to forget the picture
\[ \xymatrix@1{ \bullet \ar[r]^-{f_1} & \bullet \ar[r]^-{f_2} & \bullet \ar[r] & \cdots\ar[r] 
& \bullet \ar[r]^{f_{q-1}} & \bullet \ar[r]^{f_q} & \bullet \\} \]
of $q$ composable arrows and instead remember that the picture
\begin{equation}\label{dyslexia}
 \xymatrix@1{ x_0  & x_1 \ar[l]_-{f_{1}} & x_2 \ar[l]_-{f_{2}} & \ar[l] \cdots & x_{q-2} \ar[l] 
& x_{q-1} \ar[l]_-{f_{q-1}}    & x_q \ar[l]_-{f_q}  \\}
\end{equation}
corresponds to an element $[f_1,\cdots,f_q]$ of $N_q\sC$, so that $S(f_i)=T(f_{i+1})$.  For $x\in \sO$, we write $\id = I(x)$ generically. 
Then
\[ d_0[f] = T(f), \ \ d_1[f] = S(f), \ \ \text{and} \ \  s_0(x) = [\id_x]. \]
For $q\geq 2$, 
\[ d_i[f_1,\cdots,f_q] =\left\{
\begin{array}{ll}
[f_2,\cdots,f_q] & \mbox{if $i=0$}\\

[f_1,\cdots,f_{i-1}, f_i\com f_{i+1}, f_{i+2},\cdots,f_q]   & \mbox{if $0 < i < q$}\\

[f_1,\cdots,f_{q-1}] & \mbox{if $i=q$} \\
\end{array} \right. \]
and, for $q\geq 1$,
\[ s_i[f_1,\dots,f_q] = [f_1, \cdots, f_i, \id, f_{i+1},\cdots, f_q].  \]
Of course, these can and should be expressed in terms of the maps $S$, $T$, $I$, 
and $C$ so as to remember the topology and check continuity.

Recall that a (right) action of a group $G$ on a simplicial space $Y_*$ 
is specified by levelwise group actions such that the $d_i$ and $s_i$ are
$G$-maps; formally, $Y_*$ is a simplicial object
in the category of (right) $G$-spaces.  Orbit and fixed point 
simplicial spaces are constructed levelwise, $(Y_*/G)_q = Y_q/G$ and
$(Y_*)^{G} _q = Y_q^{G}$. For a $G$-category $\sC$, 
$N(\sC^{G})\iso (N\sC)^{G}$ since $N$ is a right adjoint, but it
is rarely the case that $N(\sC/G) \iso (N\sC)/G$, as the following
counterexample should make clear. 

\begin{exmp}\mylabel{simple} Let $G$ be a group and let $G$ act on itself 
by conjugation. Let $A$ be the abelianization of $G$. Regarding $G$ and $A$ as
categories with a single object, $G/G \iso A$, and $NA$ is generally much smaller than
$NG/G$.  Here $[g_1,\dots, g_q]$ and $[h_1,\cdots, h_q]$ are 
in the same orbit under the conjugation action if and only if there is
a single $g$ such that $g g_i g^{-1} = g h_i g^{-1}$ for all $i$. For
example if $G$ is a finite simple group of order $n$, then $A$ is
trivial but $N_qG/G$ has at least $n^{q-1}$ elements.
\end{exmp}

In this example, $NG$ is the simplicial space, often denoted $B_*G$, 
whose geometric realization is the classifying space $BG$.
Parametrizing with a left $G$-space 
$Y$ gives a familiar simplicial space $B_*(\ast,G,Y)$ (e.g. \cite[\S7]{MayClass}).
Write $q\colon E_*G \rtarr B_*G$ for the map
\[  B_*(\ast,G,G)\rtarr  B_*(\ast, G, \ast) \iso B_*(\ast,G,G)/G \]
induced by $G\rtarr\ast$. The isomorphism on the right is obvious, but it is in
fact an example of an isomorphism of the form $N(\sC/G) \iso (N\sC)/G$, as the
following observations make clear.  Recall the translation category from \myref{elt}.

\begin{lem}\mylabel{Cat1}  The simplicial space $NT(G,Y)$ is isomorphic to $B_*(\ast,G,Y)$.
\end{lem}
\begin{proof}
A typical $q$-tuple (\ref{dyslexia}) in $N_qT(G,Y)$ has $i^{th}$ term
$$ f_i =(g_i,g_{i+1}\cdots g_qy) \colon g_{i+1}\cdots g_qy\rtarr g_ig_{i+1}\cdots g_qy$$
for elements $g_i\in G$ and $y\in Y$. It corresponds to $[g_1,\cdots,g_q]y$ in $B_q(\ast,G,Y)$. 
\end{proof}

\begin{rem}\mylabel{simpcon} For any space $X$, $N\tilde X$ is the simplicial space denoted 
$D_*X$ in \cite[p. 97]{MayGeo}.  Our choice of $S$ and $T$ on $\tilde X$ is consistent with 
(\ref{dyslexia}) and the usual notation $(x_0,\cdots, x_{q})$ for $q$-simplices.  The claim in \myref{chacat} that $|N\tilde{X}|$ 
is contractible is immediate from \cite[10.4]{MayGeo}, which says that $D_*X$ is simplicially contractible. 
The isomorphism $N\mu\colon NT(G,G) \rtarr N\tilde G$ 
implied by \myref{silly} coincides with the isomorphism $\al_*\colon E_*G\rtarr D_*G$
of \cite[10.4]{MayGeo}. 
\end{rem}

Applying geometric realization, write $B(\ast,G,Y) = |B_*(\ast,G,Y)|$, and similarly
for $EG$ and $BG$. Then
$B(\ast,G,Y)\iso B(\ast,G,G)\times_G Y = EG\times_G Y$. By \myref{Cat1},
\[ B T(Y,G) = EG\times_G Y.\]
A relevant example is $Y=G/H$ for a (closed) subgroup $H$ of $G$. The space 
\[
BT(G,G/H) = EG\times_G(G/H) \iso (EG)/H
\]
is a classifying space $BH$ since $EG$ is a free contractible $H$-space. 

In particular, take $G=\sU(X,\PI)$ and $H=\PI$ for a space $X$ and
group $\PI$, remembering that $\sC\!at(\tilde X,\tilde{\PI})$ is the chaotic 
category with object space the group $\sU(X,\PI)$.  
Applying the classifying space functor to the diagram of \myref{notformal}
and using \myref{Cat1}, 
we obtain the following commutative diagram, in which the horizontal maps 
are homeomorphisms and, up to canonical homeomorphisms, the vertical maps are 
obtained by passage to orbits over $\PI$. 
\[ \xymatrix{
E(\sU(X,\PI)) \ar[r]^-{\iso} \ar[d] & B\sC\!at(\tilde{X},\tilde\PI) \ar[d]
\ar[r]^-{=} & B\sC\!at(\tilde{X},\tilde\PI) \ar[d]\\
(E\sU(X,\PI))/\PI \ar[r]_-{\iso} & B(\sC\!at(\tilde{X},\tilde\PI)/\PI) 
\ar[r]_-{\iso} & B\sC\!at(\tilde{X},\PI)\\} \]
Ignoring minor topological niceness 
conditions\footnote{The identity element of the group $\sU(X,\PI)$ should be a 
nondegenerate basepoint and the space $\sU(X,\PI)$ should be paracompact; 
see \cite[9.10]{MayClass}.}, 
for any space $X$ the diagram gives isomorphic categorical models for the universal
principal $\PI$-bundle $E\PI \rtarr B\PI$.

\section{Categorical universal equivariant principal bundles}\label{Sec3}
\subsection{Preliminaries on actions by the semi-direct product $\GA$}
Now return to the split extension (\ref{split}) of the introduction.  For a 
$\GA$-category or $\GA$-space, passage to orbits with respect to $\PI$ 
gives a $G$-category or a $G$-space.  It is standard in equivariant bundle
theory to let $G$ act from the left and $\PI$ act from the right. Thus 
suppose that $X$ is a left $G$ and right $\PI$ object in any category.
Using elementwise notation, turn the right action of $\PI$ into
a left action by setting $\si x = x\si^{-1}$. 

By an action of $\GA$ 
on $X$, we understand a left action that coincides with the given actions 
when restricted to the subgroups $G = e\times G$ and $\PI = \PI\times e$ of $\GA$.  Since
$(\si,g) = (\si,e)(e,g)$, the action must be defined by
\begin{equation}\label{action}
(\si,g)x =  (\si,e)(e,g) x = (\si,e)gx = \si gx = (gx)\si^{-1}.
\end{equation}

For now, we will denote the action of $G$ on $\Pi$ by $\cdot$, but we just use juxtaposition 
for the prescribed actions of $G$ and $\PI$ on $X$.  Since the action by $g$ on $\PI$ is a group homomorphism, $g\cdot(\si\ta) = (g\cdot \si)(g\cdot\ta)$ and 
$g\cdot\si^{-1} = (g\cdot\si)^{-1}$. 
The interaction of $\PI$ and $G$ in $\GA$ is given by the twisted commutation relation
\[ (e,g)(\si,e) = (g\cdot\si,g) = (g\cdot \si,e)(e,g), \] 
or the same relation with $\si$ replaced by $\si^{-1}$.
Therefore (\ref{action}) gives an action of $\GA$ if and only if the given actions
of $\PI$ and $G$ satisfy the twisted commutation relation
\begin{equation}\label{reaction}
g(x\si) = (gx)(g\cdot \si).
\end{equation}
The placement of parentheses is crucial: we are taking group actions in
different orders.
When the action of $G$ on $\PI$ is trivial, $g\cdot \si=\si$, this is the familiar statement that commuting left and right actions define an action by the 
product $\PI\times G$.

\begin{lem}\mylabel{action2} For a $G$-category $\sA$, the left $G$ and right $\PI$-actions on 
$\sC\!at(\sA,\tilde{\PI})$ extend naturally to a $\GA$-action.
\end{lem} 
\begin{proof} We must verify that $g(F\si) = (gF)(g\cdot\si)$ 
for $g\in G$, $\si\in \PI$ and a functor $F\colon \sA\rtarr \PI$. The unique 
natural transformation $E\rtarr F$ between a pair of functors $E$ and $F$ will 
then necessarily be given by $\GA$-maps.  The verification is formal from the 
fact that $G$ acts by conjugation, so that the action of $G$ on $\PI$ is 
part of the prescription of the action of $G$ on $F$.
Recall that the left action of $G$ on 
$\sC\!at(\sA,\tilde{\PI})$ is given by conjugation, 
$(gF)(a)=g\cdot F(g^{-1}a)$ 
for $g\in G$ and an object or morphism $a\in \sA$. The right action of $\PI$ is
given by $(F\si)(a) = F(a)\si$.  Then
\begin{eqnarray*}
(g(F\si))(a) &=& g\cdot (F\si)(g^{-1}a) \\
             &=& g\cdot (F(g^{-1}a)\si) \\ 
             &=& (g\cdot F(g^{-1}a))(g\cdot\si) \\
             &=& ((gF)(a))(g\cdot\si)\\
             &=& ((gF)(g\cdot\si))(a).   \quad \quad \qedhere
\end{eqnarray*}
\end{proof}

In particular, let $\sA = \tilde{X}$ for a left $G$-space $X$. Then the given
action of $G$ on the object space $X$ and the diagonal action of $G$ on the 
morphism space $X\times X$ give a left $G$-action on the category $\tilde{X}$.  
\myref{action2} shows that the left $G$ and right $\PI$-action on 
$\sC\!at(\tilde X,\tilde{\PI})$ give it an action by $\GA$.  Explicitly, the conjugation 
left action by $G$ and the evident right action by $\PI$ on the 
object space $\sU(X,\PI)$ induce diagonal actions on the morphism space 
$\sU(X,\PI)\times \sU(X,\PI)$, and these specify left $G$ and right $\PI$-actions 
on $\sC\!at(\tilde X,\PI)$ that satisfy the commutation relation required for 
a $\GA$-action. 

Specializing further to $X=G$, we have the
following equivariant elaboration of \myref{silly}.  We change the group $G$ there to
the group $\sU(G,\PI)$ here and remember that the product on $\sU(G,\PI)$ is just the
pointwise product induced by the product on $\PI$, with no dependence on the product
of $G$. Ignoring the group action, we may identify the chaotic right $\sU(G,\PI)$-category 
with object space $\sU(G,\PI)$ with the category $\sC\!at(\tilde G,\tilde \PI)$. The following
lemma identifies group actions. Remember that $\PI$
is a subgroup of $\sU(G,\PI)$.  

\begin{lem}\mylabel{silly2} The isomorphism of right $\sU(G,\PI)$-categories
\[ \mu\colon T(\sU(G,\PI),\sU(G,\PI))
\rtarr \Ch \PI) \]
is an isomorphism of $\GA$-categories, where the $G$-action on both
source and target categories is given by the conjugation action on 
the object space $\sU(G,\PI)$ and the resulting diagonal action on the morphism 
space $\sU(G,\PI)\times \sU(G,\PI)$. 
\end{lem}
\begin{proof} Since $\mu$ is an isomorphism and a $\PI$-map, we can and must 
give the source category the unique $G$-action such that $\mu$ is a $G$-map.  
Since $\mu$ is the identity map on object spaces, the action must be the 
conjugation action on the object space. On an element
$(\be,\al)$ of the morphism space, we must define
\[ g(\be,\al) = \mu^{-1}(g\mu(\be,\al)) = \mu^{-1}(g(\be\al),g\al) 
= \mu^{-1}((g\be)(g\al),g\al) = (g\be,g\al). \quad \qedhere \]
\end{proof}

\begin{lem}\mylabel{silly3} With $X=G$, the diagram of \myref{notformal} is a
commutative diagram of $\GA$-categories and maps of $\GA$-categories, where
$\GA$ acts through the quotient homomorphism $\GA\rtarr G$ on the three
categories on the bottom row.
\end{lem} 
\begin{proof} Since the trapezoid is obtained by passing to orbits under the
action of $\PI$, it
is a diagram of $\GA$-categories by \myref{silly2}.  The functor
$p\colon \tilde{\PI}\rtarr \PI$ of \myref{psqs} is a $G$-map since
\[  g\cdotp(\ta,\si) = g\cdot (\ta\si^{-1}) = (g\cdot\ta)(g\cdot \si)^{-1}
= p(g\cdot \ta,g\cdot \si).\]
It follows that the right vertical arrow $q = \sC\!at(\tilde G,p)$ is a
map of $\GA$-categories.   Letting $[F]$ denote the orbit of a functor
$F\colon \tilde G\rtarr \tilde \PI$ under the right action of $\PI$, the
functor $\xi$ is specified by $\xi[F] = p\com F$, and it follows that 
$\xi$ is $\GA$-equivariant. 
\end{proof}

\subsection{Universal principal $(G,\PI_G)$-bundles}\label{Universal}
Observe that for any $G$-category $\sA$, the corepresented functor 
$\sC\!at(\sA,-)$ from $G$-categories to $G$-categories is a right 
adjoint and therefore preserves all limits.  We take $\sA$ to be the 
$G$-category $\tilde{G}$ from now on, and we use the functor 
$\Ch -)$ to obtain a convenient categorical description of 
universal principal $(G,\PI_G)$-bundles. Variants of the construction are
given in \cite{May, MS}. 

\begin{defn}\mylabel{GPI} Let $G$ and $\PI$ be topological groups and let 
$G$ act on $\PI$.  Define $E(G,\PI_G)$ to be the $\GA$-space
$B\Ch \tilde\PI) = |N\Ch \tilde{\PI})|$
and define $B(G,\PI_G)$ to be the orbit $G$-space $E(G,\PI_G)/\PI$. Let
$p\colon E(G,\PI_G) \rtarr B(G,\PI_G)$
be the quotient map.
\end{defn}

We need a lemma in order to prove that $p$ is a universal $(G,\PI_G)$-bundle
in favorable cases.  We defer the proof to the next section. We believe that 
the result is true more generally, but there are point-set topological issues 
obstructing a proof. We shall not obscure the simplicity 
of our work by seeking maximum generality. As usual in equivariant bundle 
theory, we assume that all given subgroups are closed. 

\begin{lem}\mylabel{FixedUni} Let $\LA$ be a subgroup of $\GA$. If
$\LA\cap \PI\neq e$, then the fixed point category $\Ch \tilde{\PI})^{\LA}$ is empty. 
At least if $G$ is discrete, if 
$\LA\cap \PI = e$, then $\Ch \tilde{\PI})^{\LA}$ is non-empty and chaotic.
\end{lem}

The following result is \cite[Thm. 9]{LM}, but the details of the proof are 
in \cite[\S2]{La}.  A principal $(G,\PI_G)$-bundle is numerable if it is
trivial over the subspaces of $B$ in a numerable open cover.

\begin{thm}\mylabel{LaMay} A numerable principal $(G,\PI_G)$-bundle 
$p\colon E\rtarr B$ is universal if and only if $E^{\LA}$ is contractible 
for all (closed) subgroups $\LA$ of $\GA$ such that $\LA\cap \GA =\{e\}$.  
\end{thm}

We comment on the hypotheses. Recall from  point-set topology that a space 
$X$ is completely regular if for every closed subspace $C$ and every point $x$ 
not in $C$, there is a continuous function $f\colon X \rtarr [0,1]$ such that 
$f(x)=0$ and $f(C)={1}$. This is a weak condition that is satisfied by 
reasonable spaces, such as CW complexes.

\begin{rem} Specializing \cite[Propositions 4 and 5]{LM}, a 
principal $(G,\PI_G)$-bundle with completely regular total space is locally 
trivial, and a locally trivial principal $(G,\PI_G)$-bundle over a paracompact base 
space (such as a CW complex) is numerable. Therefore, modulo weak point-set topological 
conditions, the fixed point condition in \myref{LaMay} is the essential criterion for a 
universal bundle.
\end{rem}

Therefore \myref{FixedUni} has the following consequence. 

\begin{thm}\mylabel{universal} If $G$ is discrete and $\PI$ is either
discrete or a compact Lie group, the map 
\[ p\colon E(G,\PI_G)\rtarr B(G,\PI_G) \]
obtained by passage to orbits over $\PI$ is a universal principal $(G,\PI_G)$-bundle. 
\end{thm}

The classifying space $B(G,\PI_G) = |N\Ch \tilde{\PI})|/\PI$ is obtained by first
applying the classifying space functor and then passing to orbits. On the other hand,
the space $B\Ch \PI)=|N\Ch \PI)|$ is obtained by first passing to orbits on the categorical level 
and then applying the classifying space functor.  The category $\Ch \PI)$ is 
thoroughly understood, as explained in \S\ref{Sec2}. The key virtue of our model
for $B(G,\PI_G)$ is that these two $G$-spaces can be identified, 
by \myref{notformal}.

\begin{thm}\mylabel{Main} The canonical map
\[ B(G,\PI_G) = |N\Ch \tilde{\PI})|/\PI \rtarr 
|N\Ch \PI)| = B \Ch \PI) \]
is a homeomorphism of $G$-spaces.  Therefore, if $G$ is discrete and 
$\PI$ is either discrete or a compact Lie group, the map
\[ Bq\colon B\Ch \tilde\PI) \rtarr B\Ch \PI) \]
is a universal principal $(G,\PI_G)$-bundle.
\end{thm}

\begin{sch}\mylabel{scholium} For finite groups $G$, this result
is claimed in \cite[p. 1294]{MS}.  For more general groups $G$, 
\cite[3.1]{MS} states an analogous result, but with $\tilde{\PI}\rtarr \PI$
replaced by a functor defined in terms of the nonequivariant universal bundle 
$E\PI\rtarr B\PI$, resulting in a much larger construction. The replacement is 
needed for the proof of their analogue \cite[3.3]{MS} of our \myref{FixedUni}. 
A commutation relation of the form $N(\sC/\PI) = (N\sC)/\PI$ for their larger
construction is stated (five lines above \cite[3.1]{MS}), but there is 
no hint of a proof or of the need for one. It is not altogether clear to us that 
the commutation relation stated there is true, and we view
the commutation relation \myref{notformal} as the main point of the proof 
of \myref{Main}. Nevertheless, \cite{MS} had the insightful right idea that led 
to our work.
\end{sch} 

\section{Determination of fixed points}\label{Sec4}
\subsection{The fixed point spaces of $E(G,\PI_G)$}
We must prove \myref{FixedUni}. Since $\PI$ acts freely on $\Ch \tilde{\PI})$, 
it is clear that $\Ch \tilde{\PI})^{\LA}$ is empty if $\LA\cap \PI\neq e$.  
Thus assume that $\LA\cap\PI = e$.  By \myref{chaochao2}, the fixed point category 
$\Ch \tilde{\PI})^{\LA}$ is chaotic. It remains to prove that it is non-empty, 
and Lemma \ref{chaosadj} implies that this is so if and only if the space 
$\sU(G,\Pi)^{\LA}$ is non-empty.  Thus it suffices to show that $\sU(G,\Pi)$
has a $\LA$-fixed point, which means that there is a $\LA$-map $f\colon G\rtarr \PI$.
We prove this using the following standard generalization of a homomorphism and
a variant needed later.

\begin{defn}\mylabel{crossedhom} A map $\al\colon G\rtarr \PI$ is a crossed 
homomorphism if
\begin{equation}\label{cross1} \al(gh)  = \al(g) (g\cdot \al(h))
\end{equation}
for all $g,h\in G$. In particular, 
\begin{equation}\label{cross2}
\al(e) = e, \   \al(g)^{-1}=g\cdot \al(g^{-1}) \ \text{ and } \ 
\al(g^{-1})^{-1} = g^{-1}\cdot \al(g). 
\end{equation}
A map $\al\colon G\rtarr \PI$ is a crossed anti-homomorphism if
\begin{equation}\label{cross3}
\al(gh) = (g\cdot \al(h))\al(g).
\end{equation}
\end{defn}

Note that we require $\al$ to be continuous in our general topological context.

\begin{lem}\mylabel{crosssubgp}  All subgroups $\LA$ of $\GA$ such that 
$\LA\cap \PI = e$ are of the form 
\[ \LA_{\al}=\{(\al(h),h)|h\in H\},\]
where $H$ is a subgroup of $G$ and $\al\colon H\rtarr \PI$ is a crossed 
homomorphism. At least if $G$ is discrete or $\GA$ is compact, $\al$ is
continuous.
\end{lem} 
\begin{proof}
Clearly $\LA_{\al}$ is a subgroup of $\GA$ such that $\LA_{\al}\cap \PI = e$.  
Conversely, let $\LA\cap \PI=e$.  Define $H$ to
be the image of the composite of the inclusion $\io\colon \LA\subset \GA$ and the 
projection $\pi\colon \GA\rtarr G$. Since $\LA\cap \PI=e$, the composite $\pi\com \io$ is 
injective and so restricts to a continuous isomorphism $\nu\colon \LA\rtarr H$. 
For $h\in H$, define 
$\al(h)=\si$, where $\si$ is the unique element of $\PI$ such that $(\si, h)\in \LA$. 
Thus $\al$ is the composite of $\io\com \nu^{-1}\colon H\rtarr \GA$ and 
the projection $\rh\colon \GA\rtarr \Pi$. If $G$ is discrete or if $\GA$ and therefore 
$\LA$ is compact, then $\nu$ is a homeomorphism and $\al$ is continuous. For $h,k\in H$,
$$(\al(h),h)(\al(k),k)=(\al(h)(h\cdot \al(k)), hk) \in \LA, $$ so 
$\al(hk)= \al(h)(h\cdot \al(k))$. Thus $\al$ is a crossed homomorphism 
and $\LA=\LA_\al$.
\end{proof}

\begin{proof}[Proof of \myref{FixedUni}] We must obtain a $\LA$-map 
$f\colon G\rightarrow \Pi$, where $\LA=\LA_{\al}$ for a crossed 
homomorphism $\al$. By the definition of the action by $\LA$, this means 
that 
$$ f(g) = (h\cdot f(h^{-1}g))\al(h)^{-1} $$
or equivalently 
$$ h\cdot f(h^{-1}g) = f(g)\al(h) $$
for all $h\in H$ and $g\in G$. We choose right
coset representatives $\{g_i\}$ to write $G$ as a disjoint
union of cosets $Hg_i$. We then define $f:G\rtarr \Pi$ by
$$f(kg_i)=\al(k)^{-1}$$ for $k\in H$. 
By using (\ref{cross1}), writing out the inverse of a product as the product 
of inverses, using that $h^{-1}\cdot$ and $h\cdot$ are group homomorphisms
and that $\cdot$ is a group action, and finally using (\ref{cross2}) and, again, 
that $\cdot$ is a group action, we see that
\begin{eqnarray*}
h\cdot f(h^{-1}kg_i) & = & h\cdot \al(h^{-1}k)^{-1}\\
&=& h\cdot(\al(h^{-1}) (h^{-1}\cdot \al(k))^{-1}\\
&=& h\cdot ((h^{-1}\cdot\al(k))^{-1} (\al(h^{-1}))^{-1})\\
&=& (h\cdot (h^{-1}\cdot \al(k)^{-1})( h\cdot (\al(h^{-1})^{-1})\\
&=& \al(k)^{-1} (h\cdot (h^{-1}\cdot\al(h))) \\
&= & f(kg_i)\al(h). 
\end{eqnarray*}
for all $h\in H$. Thus $f$ is a $\LA$-map.  We have assumed that $G$ is
discrete in order to ensure that $f$ is continuous.
\end{proof}

\begin{rem} If we relax the condition that $G$ is discrete, we do not see how 
to prove that $f$ is continuous, as would be needed for a more general result.
\end{rem}

\subsection{The fixed point categories of $\Ch \PI)$}\label{Sec4a}

For $H\subset G$, the structure of the fixed point space $B(G,\PI_G)^H$ is 
known (up to homotopy), for example by specialization of more general results in  \cite{LM}. 
We show here how to see that structure on the category level.  In fact, we identify the fixed point 
categories  $\Ch \PI)^H$, with no restrictions on $\PI$ and $G$.
Since the functor $B$ commutes with fixed points, this gives a categorically precise 
interpretation of the fixed point space $B(G,\PI_G)^H$. 

We return to \S\ref{Sec2}, taking $X=G$ there. The $H$-fixed functors and $H$-natural transformations 
in $\Ch \PI)$ are the $H$-equivariant functors and natural
transformations, in accord with our notational convention
$\Ch \PI)^H = H\Ch \PI)$.
Since $\tilde{G}$ and $\tilde{H}$ are both $H$-free contractible categories, 
they are equivalent as $H$-categories. Therefore 
\begin{equation}\label{GHFix}
\Ch \PI)^H \simeq \sC\! at(\tilde{H},\PI)^H = H\sC\!at(\tilde{H},\PI).
\end{equation}

This implies that we may restrict to the case $G=H$ and deduce conclusions in general.
The objects and morphisms of $G\Ch \PI)$ are the $G$-equivariant 
functors $E: \tilde{G} \rightarrow \PI$ and the $G$-equivariant natural transformations
$\et$.  In \myref{Ob(X,P)}, we described a functor $E$ in terms of the map
$\al\colon G\rtarr \PI$ defined by $\al(h) = E(h,e)$.  

\begin{lem} The $G$-action on functors $E\colon \tilde{G}\rtarr \PI$ induces the 
$G$-action on maps $\al\colon G\rtarr \PI$ specified by 
\[ (g\al)(h) = (g\cdot (\al(g^{-1}h))(g\cdot \al(g^{-1})^{-1})). \]
\end{lem}
\begin{proof}
\[ (gE)(h,e) = g\cdot E(g^{-1}h,g^{-1}) = g\cdot (E(g^{-1}h,e)E(e,g^{-1})). \qedhere \]
\end{proof}

\begin{lem} The space of objects of $G\Ch \PI)$ can be identified
with the subspace of $\sU(G,\PI)$ consisting of the crossed anti-homomorphisms 
$\al\colon G\rtarr \PI$.
\end{lem}
\begin{proof} Setting $g\al = \al$ and applying $g^{-1}\cdot (-)$ to the formula for
the action of $G$ on $\al$, we obtain
\[ g^{-1}\cdot \al(h) = \al(g^{-1}h)\al(g^{-1})^{-1}. \]
Replacing $g^{-1}$ by $g$ and multiplying on the right by $\al(g)$, this gives
$$ \al(gh) = (g\cdot \al(h))\al(g)$$ 
for all $g,h\in G$, which says that $\al$ is a crossed anti-homomorphism.
\end{proof}

Similarly, as in \myref{Mor(X,P)}, a natural transformation
$\eta \colon  E_{\al}\rtarr E_{\be}$ 
is determined by $\si = \eta(e)$. Explicitly,
$$\et(g) = E_{\be}(g,e)\et(e) E_{\al}(g,e)^{-1} = \be(g)\si \al(g)^{-1}$$
for $g\in G$. Now a $G$-fixed natural transformation $\et$ satisfies 
$\eta(gh) =g\cdot \eta(h)$ for $g,h \in G$ and thus 
$\eta(g)=\eta(ge)= g \cdot \eta(e) = g\cdot \si$. Therefore the naturality 
square for $G$-fixed natural transformations translates into 
$$ g\cdot \si = \be(g)\si \al(g)^{-1}$$
or equivalently 
\begin{equation}\label{natnat}
\be(g)\si = (g\cdot \si) \al(g).
\end{equation}

We use the following definitions and lemma to put things together.
 
\begin{defn}\mylabel{crossedhom2}
Let $G$ act on $\PI$. Define the crossed functor 
category $\sC at_\times (G,\PI )$ to be the category whose objects are the crossed 
homomorphisms $G\rtarr \PI$ and whose morphisms $\si\colon \al\rtarr \be$ are the elements 
$\si\in \PI$ such that $ \be(g)(g\cdot\si) = \si\al(g)$; they are
are called isomorphisms of crossed homomorphisms. The
composite $\ta\com\si$, $\ta\colon \be\rtarr \ga$ is given by $\ta\si$. Define 
the centralizer $\Pi^\al$ of a crossed homomorphism $\al\colon G\rtarr \Pi$ to
be the subgroup 
$$ \Pi^\al = \{\si \in \PI| \al(g) (g\cdot \si) = \si \al(g) \ \ \mbox{for all $g\in G$}\} $$
of $\PI$.  It is the automorphism group $\Aut(\al)$ of the object
$\al$ in $\sC at_\times (G,\PI )$.  
\end{defn}

\begin{defn} Define the anti-crossed functor
category $\sC at_\times ^{-}(G,\PI )$
to have objects the crossed anti-homomorphisms $\al\colon G\rtarr \PI$ and morphisms
$\si\colon \al\rtarr \be$ the elements $\si\in \PI$ such that 
$\be(g)\si = (g\cdot\si)\al(g)$, with $\tau\com \si = \ta\si$. The centralizer 
$\PI^{\al}$ of a crossed 
anti-homomorphism $\al\colon G\rtarr \PI$ is 
$$ \Pi^\al = \{\si \in \PI| \al(g)\si = (g\cdot \si) \al(g) \ \ \mbox{for all $g\in G$}\}. $$
Again, $\Pi^{\al} = \Aut(\al)$ in $\sC at_\times ^{-}(G,\PI )$.
\end{defn}

If the action of $G$ on $\PI $ is trivial, then the crossed functor 
category is just the functor category $\sC\! at(G,\PI )$ since homomorphisms 
$\al\colon G\rightarrow \PI$ correspond to functors $\al\colon G\rtarr \PI$ 
and elements $\si\in \PI$ such that $\be(g)\si = \si\al(g)$ for $g\in G$ correspond 
to natural transformations $\al\rtarr \be$. In that case, 
$$\PI^{\al} = \{\si \in \PI| \si^{-1}\al(g) \si =\al(g) \ \ \mbox{for all $g\in G$}\}$$
is the usual centralizer of $\al$ in $\PI$, and then the following identification is obvious.  

\begin{lem} The categories $\sC at_\times (G,\PI )$ and 
$\sC at_\times^{-} (G,\PI)$ of crossed homomorphisms
and crossed anti-homomorphisms are canonically isomorphic.
\end{lem}
\begin{proof}  For a crossed homomorphism $\al\colon G\rtarr \PI$, define 
$\bar{\al}\colon G\rtarr \PI$ by 
$$\bar{\al}(g) = g\cdot\al(g^{-1}).$$  Then
\[ \bar{\al}(gh) =  (gh)\cdot\al(h^{-1}g^{-1}) 
= g\cdot h\cdot(\al(h^{-1})(h^{-1}\cdot \al(g^{-1})) 
= (g\cdot \bar{\al}(h))(\bar{\al}(g)),\]
so that $\bar{\al}$ is a crossed anti-homomorphism.  If $\si$
is a morphism $\al\rtarr \be$ in $\sC at_\times (G,\PI )$, then 
$\be(g)(g\cdot\si) = \si\al(g)$. It follows that
\[ \bar\be(g)\si = (g\cdot \be(g^{-1}))\si
= g\cdot (\be(g^{-1})(g^{-1}\cdot\si)) = g\cdot (\si\al(g^{-1})) 
= (g\cdot\si)\bar{\al}(g),\]
so that $\si$ is also a morphism $\bar\al\rtarr \bar\be$ in 
$\sC at_\times^{-} (G,\PI )$. The construction of the inverse isomorphism is similar.
\end{proof}

Returning to the $G$-fixed category of interest, we summarize our discussion in 
terms of these definitions and results.

\begin{thm}
The fixed point category $G\Ch \PI) = \Ch \PI)^G$ is isomorphic to the
anti-crossed functor category $\sC\! at_\times^- (G, \PI)$. Therefore
it is also isomorphic to the crossed functor category 
$\sC\! at_\times (G, \PI)$.
\end{thm}

\begin{cor}\mylabel{fixedcor} For $H\subset G$, the fixed point category 
$\Ch \PI)^H$ is equivalent to the anti-crossed functor category $\sC\! at_\times^- (H, \PI)$. 
Therefore it is also equivalent to the crossed functor category $\sC\! at_\times (H,\PI)$.
\end{cor}

\begin{rem} The appearance of anti-homomorphisms in this context is not new;
see e.g. \cite{usenko}.  As we have seen, it is also innocuous. We have chosen
not to introduce opposite groups, but the anti-isomorphism 
$(-)^{-1}\colon \PI \rtarr \PI^{op}$ is relevant.
\end{rem} 

\subsection{Fixed point categories, $H^1(G;\PI_G)$, and Hilbert's Theorem 90}

Since $G\Ch \PI)$ is a groupoid, it is equivalent to the 
coproduct of its subcategories $Aut(\al)$, where we choose one $\al$ from
each isomorphism class of objects.  The following definition is standard when $\PI$
and $G$ are discrete but makes sense in general.

\begin{defn}\mylabel{crossedhom3} 
The first non-abelian 
cohomology group $H^1(G;\PI_G)$ is the pointed set of isomorphism 
classes of crossed homomorphisms $G\rtarr \PI$.  We write $[\al]$ for the 
isomorphism class of $\al$. The basepoint of $H^1(G;\PI_G)$ is $[\epz]$, 
where $\epz$ is the
trivial crossed homomorphism given by $\epz(g) = e$ for $g\in G$.
\end{defn}

With this language, (\ref{GHFix}) and \myref{fixedcor} can be restated as follows. 

\begin{thm}\mylabel{FFFxxx} For $H\subset G$, $\Ch \PI)^H$ 
is equivalent to the coproduct of the categories $\Aut(\al)$, where the coproduct runs over 
$[\al]\in H^1(H;\PI_H)$.
\end{thm}  

Here $\Aut(\al)$ implicitly refers to the ambient group $\PI\rtimes H$, not 
$\GA = \PI\rtimes G$.  By (\ref{GHFix}) or, more concretely, \myref{finlem2} below,
we obtain the same group $\Aut(\al)$ for $\al$ considered as an object of 
$\sC\!at(\tilde K,\PI)^H$ for any $H\subset K\subset G$. 

For any $G$-category $\sA$, we have a natural map of $G$-categories
\[ \io\colon \sA\rtarr \Ch \sA). \]
It is induced by the the unique $G$-functor $\tilde{G}\rtarr \ast$,
where $\ast$ is the trivial $G$-category with one object and its identity morphism.
The $G$-fixed point functor $\io^G$ played a central role in Thomason's paper \cite{Thom}.
When $\sA = \PI$ for a $G$-group $\PI$, $\io$ sends the unique object of 
$\PI$ to the basepoint $[\epz]\in H^1(G;\PI)$.   

We shall describe the groups $\Aut(\al)$ in familiar group theoretic terms in
the next section.  As a special case, $\Aut(\epz) = \PI^G$ and $\io^G$ restricts 
to the identity functor from $\PI^G$ to $Aut(\epz)$.  This implies the following result.

\begin{prop} The functor $\io^G\colon \PI^G \rtarr \Ch \PI)^G$
is an equivalence of categories if and only if
$H^1(G;\PI_G) = [\epz]$.
\end{prop}

\begin{exmp}\mylabel{Serre} Let $E$ be a Galois extension of a field $F$ with Galois group $G$.
Then $G$ acts on $E$ and $E^G = F$.  Let $G$ act entrywise on $GL(n,E)$.  Then
Serre's general version of Hilbert's Theorem 90 \cite[Ch 10, Prop. 3]{Serre} gives 
that $H^1(G;GL(n,E)_G) = [\epz]$. Since $GL(n,E)^G = GL(n,F)$, we conclude that 
$\io^G$ is an equivalence of categories
\[  GL(n,F)\rtarr \Ch GL(n,E))^G. \]
More generally, for $H\subset G$, $\io^H$ is an equivalence of categories
\[  GL(n,E^H)\rtarr \Ch GL(n,E)^H. \]
\end{exmp}

As explained in \cite[\S4.5]{GM} this gives precisely the information 
that ensures that the algebraic $K$-theory fixed point spectrum $\bK_G(E)^H$ is 
equivalent to  $\bK(E^H)$.  We shall return to consideration of $G$-rings 
such as $E$ in \S6.

We recall the easy calculation of $H^1(G;\PI)$ in group theoretic terms.

\begin{lem}\mylabel{finlem1}  The set $H^1(G;\PI)$ is in
bijective correspondence with the set of $\PI$-conjugacy classes of 
subgroups $\Lambda$ of $\GA$ such that $\Lambda \cap \Pi=e$ and 
$q(\Lambda)=G$.
\end{lem}
\begin{proof}
By \myref{FixedUni}, the subgroups $\Lambda$ of $\GA$ such that 
$\Lambda\cap \Pi=e$ are of the form
$$\Lambda_\al= \{(\al(h), h)| h\in H\}$$ 
for a crossed homomorphism $\al\colon H\rightarrow \Pi$. 
If $\si\in \PI$, then $\si \LA_{\al}\si^{-1}\cap \PI = e$ and therefore 
$\si \LA_{\al}\si^{-1} = \LA_{\be}$ for some crossed homomorphism $\be$.
The equality forces $\be$ and $\al$ to be defined on the same subgroup $H$ 
and to satisfy 
$\be(g)(g\cdot\si) = \si\al(g)$.
We are concerned only with the case $H=G$, and then this says that
$\si$ is a morphism and thus an
isomorphism $\al\rtarr \be$ in $\sC\!at_{\times}(G,\PI)$.
\end{proof}

\subsection{The fixed point spaces of $B(G,\PI_G)$}

We here identify the automorphism groups $\Aut(\al)$ group theoretically 
and so complete the identification of $\Ch \PI)^G$.

\begin{lem}\mylabel{finlem2}  Let $\al\colon H\rtarr \PI$ be a crossed homomorphism
and $\PI$ be a $G$-group, where $H\subset G$. Then the crossed centralizer $\Pi^{\al}$ is the 
intersection $\PI\cap N_\GA \Lambda_\al$.  Therefore this intersection is the
same for all $\GA_K = \PI\rtimes K$, $H\subset K\subset G$. 
\end{lem}
\begin{proof}Let $(\pi,g)\in \PI\rtimes G$ and $h\in H$. Calculating in 
$\GA = \PI\rtimes G$, we have
\begin{eqnarray*}
(\si,g)^{-1}(\al(h),h)(\si,g)
&=& (g^{-1}\cdot \si^{-1}, g^{-1})(\al(h),h)(\si,g) \\
&=& ((g^{-1}\cdot \si^{-1})( g^{-1}\cdot \al(h)) , g^{-1}h)(\si,g) \\ 
&=& ((g^{-1}\cdot \si^{-1})(g^{-1}\cdot \al(h))((g^{-1}h)\cdot \si), g^{-1}hg).
\end{eqnarray*}
Therefore $(\si,g)$ is in $N_{\GA}\LA_{\al}$ if and only if $g$ is in $N_GH$ and
$$ \al(g^{-1}hg) = (g^{-1}\cdot \si^{-1})(g^{-1}\cdot \al(h))((g^{-1}h)\cdot \si)$$
for all $h\in H$. 
When $g=e$, so that $\si = (\si,e)$ is a typical element of $\PI\cap N_{\GA}\LA_{\al}$,
this simplifies to
$$\al(h) = \si^{-1}\al(h)(h \cdot \si). \qedhere$$
\end{proof}

Passing to classifying spaces from \myref{FFFxxx} gives the following result.

\begin{thm}\mylabel{FixFix} For $H\subset G$,
\[ B(G,\PI_G)^H = B\Ch \PI)^H \htp \coprod B\Aut(\al), \]
where the coproduct runs over $[\al]\in H^1(H;\PI_H).$ 
\end{thm}

By Lemmas \ref{finlem1} and \ref{finlem2}, we can restate \myref{FixFix} as follows. 

\begin{thm}\mylabel{FixToo}
Let $\GA=\Pi \rtimes G$. For a subgroup $H$ of $G$,
$$B(G,\PI_G)^H \htp \coprod B(\Pi\cap N_{\GA}\LA),$$
where the union runs over the $\Pi$-conjugacy classes of subgroups 
$\Lambda$ of $\GA$ such that $\Lambda \cap \Pi=e$ and $q(\Lambda)=H$.
\end{thm}

Of course, we are only entitled to consider $B(G,\PI_G)$ as a classifying
space for principal $\GA$-bundles when \myref{Main} applies.  The fixed point 
spaces $B(\Pi; \Gamma)^H$ of classifying spaces are studied more generally in 
\cite{LM} when $\GA$ is given by a not necessarily split extension of compact Lie 
groups 
\begin{equation}\label{nonsplit}
\xymatrix@1{
1\ar[r] & \Pi \ar[r] & \Gamma \ar[r]^-{q} &  G \ar[r] & 1.}\\
\end{equation} 
For such groups $\GA$, \cite[Theorem 10]{LM} gives an entirely different
bundle theoretic proof that the conclusion of \myref{FixToo} still holds, exactly
as stated. When \cite{LM} was written, no particularly nice model for the 
homotopy type $B(\PI;\GA)$ was known.

\section{The comparison between $B\Ch \PI)$ and $\sU(EG,B\PI)$}\label{Sec5}
A convenient model $p\colon E(\PI;\GA) \rtarr B(\PI;\GA)$ for a universal principal 
$(\PI;\GA)$-bundle was later given in terms of mapping spaces \cite{May}. 
Start with the classical models in \S\ref{Secnerve} for universal principal $\PI$, $G$, 
and $\GA$-bundles and let $Eq\colon E\GA\rtarr EG$ be the map induced by the quotient
homomorphism $q\colon \GA\rtarr G$.  Let $\Sec(EG,E\GA)$ denote the $\GA$-space of 
sections $f\colon EG\rtarr E\GA$, so that $Eq\com f = \id$.  The following result 
is part of \cite[Theorem 5]{May}.

\begin{thm}\mylabel{Mayone} The quotient map 
$p\colon \Sec(EG,E\GA)\rtarr \Sec(EG,E\GA)/\PI$
is a universal principal $(\PI;\GA)$-bundle.
\end{thm}

Now let the extension be split, so that $\GA = \PI \rtimes G$.\footnote{Ignoring 
group actions, the spaces $E\GA$ and $E\PI\times EG$ are certainly 
homeomorphic since both are homeomorphic to the classifying space of the 
chaotic category with object space $\PI\times G$.} The given 
action of $G$ induces a left action of $G$ on $E\PI$ that, together with the
free right action by $\PI$, makes it a $\GA$-space.  Taking $EG$ to be a left
$G$-space and letting $\GA$ act through $q$ on $EG$, we have the product 
$\GA$-space $E\PI\times EG$. It is free as a $\GA$-space because $E\PI$ is
free as a $\PI$-space and $EG$ is free as a $G$-space.  Since it is obviously
contractible, we may as well take $E\GA = E\PI\times EG$.  Since the second
coordinate of a section $f\colon EG\rtarr E\PI\times EG$ must be the identity,
we then have
\[ \Sec(EG,E\GA) = \sU(EG,E\PI). \]
Its $\GA$-action is defined just as was the $\GA$-action on 
$\Ch \PI)$
in \myref{action2}.  This gives the following specialization of
\myref{Mayone}, which is the space level forerunner of the categorical \myref{universal}. 

\begin{thm}\mylabel{universal2} The quotient map 
$p\colon \sU(EG,E\PI)\rtarr \sU(EG,E\PI)/\PI$
is a universal principal $(G,\PI_G)$-bundle.
\end{thm} 

We also have the mapping space $\sU(EG,B\PI)$.  The canonical map $E\PI\rtarr B\PI$ induces a map 
$q\colon \sU(EG,E\PI)\rtarr \sU(EG,B\PI)$.  Then there is an induced map 
$\xi$ that makes the following diagram commute.
\[ \xymatrix{
& \sU (EG,E\PI) \ar[dl]_{p} \ar[d]^{q}  \\
\sU  (EG,E\PI)/\PI \ar[r]_-{\xi} &  \sU  (EG,B\PI).} \\ \]
The analogy with the triangle in \myref{notformal} should be evident.  As
observed in \cite[Theorem 5]{May}, elementary covering space 
theory gives the following result, which is the space level forerunner of 
the categorical \myref{Main}.

\begin{thm}\mylabel{Main2} If $\PI$ is discrete, then 
$\xi\colon \sU(EG,E\PI)/\PI \rtarr \sU(EG,B\PI)$
is a homeomorphism and therefore $q\colon \sU(EG,E\PI)\rtarr \sU(EG,B\PI)$ is a 
universal principal $(\PI;\GA)$-bundle. 
\end{thm} 

Note that $G$ but not $\PI$ is required to be discrete in \myref{Main}, 
whereas $\PI$ but not $G$ is required to be discrete in \myref{Main2}.
There is an obvious comparison map relating the categorical and space level 
constructions. For any $G$-categories $\sA$ and $\sB$, we have the evaluation 
$G$-functor
\[ \epz\colon \sC\!at(\sA,\sB)\times \sA \rtarr \sB.\]
Applying the classifying space functor and taking adjoints, this gives a $G$-map
\begin{equation}\label{compare}
 \xi\colon B\sC\!at(\sA,\sB) \rtarr \sU(B\sA,B\sB).
\end{equation}
When $\sA$ and $\sB$ are both discrete (in the topological sense), there is a simple 
analysis of this map in terms of the simplicial mapping space $\Map(N\sA,N\sB)$.
The following two lemmas are well-known nonequivariantly.

\begin{lem} For categories $\sA$ and $\sB$, there is a natural isomorphism
\[ \mu \colon N\sC\!at(\sA,\sB) \iso \Map(N\sA,N\sB), \]
and this is an isomorphism of simplicial $G$-sets if $\sA$ and $\sB$ are $G$-categories.
\end{lem}
\begin{proof} Let $\DE_n$ be the poset $\{0,1,\cdots, n\}$, viewed as a category.
The $n$-simplices of $\sC\!at(\sA,\sB)$ are the functors $\DE_n\rtarr \sC\!at(\sA,\sB)$.
By adjunction, they are the functors $\sA\times \DE_n\rtarr \sB$. Since $N$ is
full and faithful, these functors are the maps of simplicial sets 
\[ N\sA\times N\DE_n\iso N(\sA\times \DE_n) \rtarr N\sB. \]
By definition, these maps are the $n$-simplices of $\Map(N\sA,N\sB)$. These identifications
give the claimed isomorphism of simplicial sets. The compatibility with the actions of
$G$ when $\sA$ and $\sB$ are $G$-categories is clear.
\end{proof}

\begin{lem} For simplicial sets $K$ and $L$, there is a natural map
\[ \nu\colon |\Map(K,L)|\rtarr \sU(|K|,|L|). \]
If $K$ and $L$ are simplicial $G$-sets, $\nu$ is a map of $G$-spaces, and it is a 
weak equivalence of $G$-spaces when $L$ is a Kan complex.
\end{lem}
\begin{proof}  
The evaluation map $\Map(K,L)\times K \rtarr L$ induces a map
\[ |\Map(K,L)|\times |K|  \iso |\Map(K,L)\times K| \rtarr |L| \]
whose adjoint is $\nu$.  When $L$ is a Kan complex, so is $\Map(K,L)$
(e.g. \cite[6.9]{MaySS}), 
and the natural maps $L\rtarr S|L|$ and $\Map(K,L)\rtarr S|\Map(K,L)|$ 
are homotopy equivalences, where $S$ is the total singular complex functor. 
A diagram chase
shows that $\xi$ induces a bijection on homotopy classes of maps 
\[  \xi_*\colon [|J|,|\Map(K,L)|\ \rtarr [|J|, \sU(|K|,|L|)] \]
for any simplicial set $J$. Letting $G$ act trivially on $J$, all functors
in sight commute with passage to $H$-fixed points, and the equivariant 
conclusions follow.
\end{proof}

Now the following result is immediate from the definitions and lemmas above.

\begin{prop}\mylabel{catcomp}  For $G$-categories $\sA$ and $\sB$, the map $\xi$ of (\ref{compare})
is the composite $\nu\com \mu$, and it is a weak $G$-equivalence if $\sB$ is a 
groupoid.
\end{prop}

Returning to the topological setting, take $\sA=\tilde{G}$ and write 
$EG = |N\tilde{G}|$, as we may. Recalling that $E\PI\rtarr B\PI$ is obtained 
by applying $B$ to the functor
$\tilde{\PI}\rtarr \PI$, we obtain the following commutative diagram. 
\[ \xymatrix{
B\Ch \tilde \PI) \ar[d] \ar[rr] &  & \sU(EG,E\PI) \ar[d] \\
B\Ch \tilde \PI)/\PI \ar[d] \ar[rr]  & & \sU(EG,E\PI)/\PI \ar[d] \\
B\Ch \PI)  \ar[rr] &  & \sU(EG,B\PI) \\ }  \]
Theorems \ref{universal} and \ref{universal2} say that that the top two vertical
arrows are often universal principal $(\PI;\GA)$-bundles, in which case the
top two horizontal arrows are equivalences.  Theorems \ref{Main} and \ref{Main2}
say that the lower two vertical arrows and therefore also the bottom horizontal
arrow are also often equivalences. When both $\PI$ and $G$ are discrete, the
equivalences are immediate from \myref{catcomp}. More elaborate arguments might 
prove all of these results in greater topological generality. 

\section{Other categorical models for classifying spaces $B(G,\PI_G)$}
For particular $G$-groups $\PI$, there are alternative categorical models
for universal principal $(G,\PI_G)$-bundles that are important in our 
applications in \cite{GM, Merling}.  They lead to equivalent, but more
intuitive, constructions of categorical models for a number of interesting
$G$-spectra, in particular suspension $G$-spectra and the equivariant
$K$-theory spectra of rings with actions by $G$.

Perhaps surprisingly, \color{black} the symmetric groups $\SI_n$ with trivial $G$-action are of particular 
importance in equivariant infinite loop space theory.  For a ring $R$ with 
an action of a group $G$ via ring maps, the general linear groups 
$GL(n,R)$ with $G$-action on all matrix entries are of particular importance.  
We give alternative models for universal principal bundles applicable to these cases.
We focus on the total spaces here and explain additional structure on the 
resulting classifying spaces in \cite{GM2}. We assume that $G$ is finite,
although some of the definitions make sense and are interesting more generally.

\subsection{A model $\tilde{\sE}_G(n)$ for $E(G,\SI_n)$}

\begin{defn}\mylabel{UU}
Let $U$ be a countable ambient $G$-set that contains countably many copies of each orbit $G/H$.
The action of $G$ on $U$ fixes bijections $g \colon A\rtarr gA$ for all finite subsets $A$ of $U$, denoted $a \mapsto g\cdot a$.  
\end{defn}

Let $\mathbf{n}=\{1,\cdots,n\}$ and view elements $\si\in \SI_n$ as functions
$\mathbf{n}\rtarr \mathbf{n}$, so that $\si(i) = \si\cdot i$ gives a left
action of $\SI_n$ on $\mathbf{n}$. 

\begin{defn}\mylabel{EGJ} 
For $n\geq 0$, let $\tilde{\sE}_G(n)$ denote the
chaotic $(\SI_n\times G$)-category whose set $\sO b$ of objects is the set of 
pairs $(A,\al)$, where $A$ is an $n$-element subset of $U$ and $\al\colon \mathbf{n}\rtarr A$ is a bijection. 
Let $G$ act on $\sO b$ on the left by postcomposition and 
let $\SI_n$ act on the right by precomposition.  Thus
$g(A,\al) = (gA,g\com \al)$ for $g\in G$, and $(A,\al)\si = (A,\al\com \si)$ for 
$\si\in\SI$; of course
\[(g\com \al)\com \si = g\com\al\com \si = g\com (\al\com \si).\]
The action of $\SI_n\times G$ is given by $(\si,g)(A,\al) = (gA,g\com\al\com \si^{-1})$.  
Since $\tilde{\sE}_G(n)$ is chaotic, this fixes the actions on the morphism set, 
which the map $(S,T)$ identifies with $\sO\,b\times \sO b$ with $\SI_n\times G$
acting diagonally. 
\end{defn}

\begin{prop}\mylabel{GSIHtpyType} For each $n$, the classifying space 
$|N\tilde{\sE}_G(n)|$ is a universal principal $(G,\SI_n)$-bundle. 
\end{prop}
\begin{proof} For each $A$, choose a base bijection 
$\et_A\colon \mathbf{n} \rtarr A$. The function sending $\si$ to 
$(A,\et_A\com\si)$ is an isomorphism of right $\SI_n$-sets
from $\SI_n$ to the set of objects $(A,\al)$; its inverse
sends $(A,\al)$ to $\et^{-1}_A\com \al$. Thus $\SI_n$ acts freely 
on $\tilde{\sE}_G(n)$. Since $\tilde{\sE}_G(n)$ is chaotic, it 
suffices to show that the set of objects of
$\tilde{\sE}_G(n)^{\LA}$ is non-empty if $\LA\cap\SI_n=\{e\}$. 
As usual, $\LA =\{(\rh(h),h)|h\in H\}$, where $H$ is a subgroup of $G$ and 
$\rh\colon H\rtarr \SI_n$ is a homomorphism. 

Let $H$ act through $\rh$ on $\mathbf{n}$, so that $h\cdot i = \rh(h)(i)$.
Since $U$ contains a copy of every finite $G$-set, there is a bijection 
of $G$-sets $\be\colon G\times_H \mathbf{n}\rtarr B\subset U$. Its restriction
to $\mathbf{n}$ gives a bijection of $H$-sets $\al\colon \mathbf{n}\rtarr A\subset B$. 
We claim that this $(A,\al)$ is a $\LA$-fixed object.  Obviously $hA=A$ for $h\in H$. 
By \myref{EGJ}, we have 
$(\rh(h),h)(A,\al) = (A,h\com \al\com \rh(h)^{-1})$, where
\begin{eqnarray*} 
(h\com \al\com \rh(h)^{-1})(i) & = & h\cdot \al(\rh(h)^{-1}(i))\\ 
& = & h\cdot h^{-1}\cdot \al(i) = \al(i).  \qedhere\\ 
\end{eqnarray*}
\end{proof}

\begin{defn}\mylabel{ExamE} Define $\sE_G(n)$ to be the orbit $G$-category
$\tilde{\sE}_G(n)/\SI_n$.  
\end{defn}

By \myref{GSIHtpyType} and \S2.3, $B\sE_G(n)$ is a classifying space $B(G,\SI_n)$. 
Up to isomorphism, the $G$-category $\sE_G(n)$ admits the following 
more explicit description. 

\begin{lem}\mylabel{RightE} The objects of $\sE_G(n)$ are the $n$-pointed subsets 
$A$ of $U$. The morphisms are the bijections $\al\colon A\rtarr B$,
with the evident composition and identities. The group $G$ 
acts by translation on objects and by conjugation on morphisms.
That is, $g$ sends $A$ to $gA$ and $\al$ to $g\al$, where 
$g\al = g\com \al\com g^{-1}$, so that $(g\al)(g\cdot a) = g\cdot\al(a)$. 
\end{lem}
\begin{proof} The objects $(A,\al)$ are all in the same orbit, denoted $A$,
and the bijections $\et_A$ chosen in the proof of \myref{GSIHtpyType} give 
orbit representatives for the objects of $\sE_G(n)$. In $\tilde{\sE}_G(n)$, 
we have a unique morphism $\io_{\be}\colon (A,\et_A) \rtarr (B,\be)$ for each 
bijection $\be\colon \mathbf{n}\rtarr B$, and these morphisms give orbit
representatives for the set of morphisms $A\rtarr B$ in $\sE_G(n)$.  Letting
the orbit of $\io_{\be}$ correspond to the bijection 
$\al = \be\com \et_A^{-1}\colon A\rtarr B$
and noting that $\al = \et_B\com \si\com \et_A^{-1}$ for a unique 
$\si\in \SI_n$, we obtain
the claimed description of $\sE_G(n)$. Since  $\et_A$ specifies an ordering on
$A$, $\et_{gA}$ is fixed as $g\com \et_A$. Then if $\al = \be\com \et_A^{-1}$,
\[ g\com \al\com g^{-1} = g\com (\be\com \et_A^{-1})\com (\et_A \com \et_{gA}^{-1}) 
= g\com \be \com \et_{gA}^{-1}
\colon gA \rtarr gB.   \qedhere \]
\end{proof}

\subsection{$G$-rings, $G$-ring modules, and crossed homomorphisms}

By a $G$-ring we understand a ring $R$ with a left action of $G$ on $R$ through
ring automorphisms. We do not assume that $R$ is commutative, although that 
is the case of greatest interest to us.  Following the literature, we write 
$g(r) = r^g$ for the automorphism $g\colon R\rtarr R$ determined by $g\in G$. 
Then $r^{gh} = g(h(r)) = (r^h)^g$.

When $R$ is a subquotient of $\bQ$, the only automorphism of $R$ is the identity 
and the action of $G$ must be trivial, but non-trivial examples abound.  One important 
example is the action of the Galois group on a Galois extension $E$ of a field $F$. 

In the next section we will give an analogue of $\tilde{\sE}_G(n)$ but with 
$\PI = \SI_n$ replaced by $\PI = GL(n,R)$ with the entrywise action of $G$. 
We will need a tiny bit of what appears to us to be a relatively undeveloped 
part of representation theory.

For a $G$-ring $R$, there are standard notions of a ``crossed product'' ring, a
``group-graded ring'', and, as a special case of both, a ``skew group ring'', 
variously denoted $R\rtimes G$ or $R\ast G$. We shall use 
the notation $R_G[G]$ for the last of these notions. If the action of $G$ on $R$ is 
given by the homomorphism $\tha\colon G\rtarr \Aut(R)$, a more precise notation would be 
$R_\tha[G]$.  Observe that $R$ is a $k$-algebra, where $k$ denotes the
intersection of the center of $R$ with $R^G$. 

\begin{defn} As an $R$-module, $R_G[G]$ is the same as the group ring $R[G]$, 
which is the case when $G$ acts trivially on $R$. 
We define the product on $R_G[G]$ by $k$-linear (not $R$-linear) extension 
of the relation
\[  (rg)\, (sh) = r s^g\, gh \]
for $r,s\in R$ and $g,h\in G$.  Thus $R$ and $k[G]$ are subrings of $R_G[G]$ and 
\[  g\, r = r^g\, g. \] 
\end{defn}

\begin{defn}
We call (left) $R_G[G]$-modules ``$G$-ring modules'' or ``skew $G$-modules''.  Such 
an $M$ is a left $R$-module and a left $k[G]$-module such that $g(rm) = r^g(gm)$ for $m\in M$. 
If $M$ is $R$-free, we call $M$ a skew representation of $G$ over $R$. 
\end{defn}

Although special cases have appeared and there is a substantial literature
on crossed products, group-graded rings, and skew group rings (for example 
\cite{Dade, Mont, Pass}), we have not found 
a systematic study of these representations in the literature.  Kawakubo's
paper \cite{Kawa} gives a convenient starting point. The following relationship with 
crossed homomorphisms is his \cite[5.1]{Kawa}. 

\begin{thm}\mylabel{CrossR}  Let $R$ be a $G$-ring. Then the set of isomorphism
classes of $R_G[G]$-module structures on the $R$-module $R^n$ is in canonical bijective 
correspondence with $H^1(G;GL(n,R))$. In detail, let $\{e_i\}$ be the standard basis 
for $R^n$.  Then the formula 
\[ g e_i = \rh(g)(e_i) \]
establishes a bijection between $R_G[G]$-module structures on $R^n$ and 
crossed homomorphisms $\rh\colon G\rtarr GL(n,R)$.  Moreover, two $R_G[G]$-modules
with underlying $R$-module $R^n$ are isomorphic if and only if their 
corresponding crossed homomorphisms are isomorphic.  
\end{thm}
\begin{proof} Given an $R_G[G]$-module structure on $R^n$, define
the matrix $\rh(g)$ in $GL(n,R)$ by letting its $i^{th}$ column be $(s_{i,j})$,
where  
\[ g e_i = \sum_j s_{i,j} e_j. \]
Conversely, given $\rh$, write $\rh(g) = (s_{i,j})$ and define $ge_i$
by the same formula. From either starting point, we have $g e_i = \rh(g)(e_i)$.
For a second element
$h\in G$, write $\rh(h) = (t_{i,j})$, where $\rh(h)$ is either determined by an 
$R_G[G]$-module structure or is given by a crossed homomorphism $\rh$.  
Since $g\,r = r^g\, g$ in $R_G[G]$ and $g\, (r_{i,j}) = (r_{i,j}^g)$ in
$GL(n,R)$, the relation $(gh) e_i = g(h e_i)$ required of 
an $R_G[G]$-module is the same as the relation 
$\rh(g)\rh(h)(e_i) = \rh(g) (g\rh(h))(e_i)$ 
required of a crossed homomorphism.  Indeed, $(gh) e_i = \rh(gh)(e_i)$ and
\begin{eqnarray*}  
g (he_i) & = & g \rh(h)(e_i) = \sum_j \, g(t_{i,j}e_j) \\
& = & \sum_j\, t_{i,j}^g ge_j = \sum_j \sum_k t_{i,j}^g s_{j,k} e_k \\
& = & \rh(g)(\sum_j\, t_{i,j}^g e_j) = \rh(g) (g \rh(h)(e_i). \\
\end{eqnarray*}
The remaining compatibilities, in particular for the transitivity relation 
required of a module, are equally straightforward verifications, as is the
verification of the statement about isomorphisms.
\end{proof}  

The following easy observation specifies the permutation skew representations.  
For a set $A$, let $R[A]$ denote the free $R$-module on the basis $A$. 

\begin{prop}\mylabel{permrep}  Let $A$ be a $G$-set and define
\[  g ( \sum_a\,r_a a) =  \sum_a r^g_a ga \]
for $g\in G$, $r_a\in R$, and $a\in A$. Then $R[A]$ is an $R_G[G]$-module.
\end{prop}

In view of  \myref{CrossR}, this has the following immediate consequence.

\begin{cor}  For a $G$-ring $R$, any $n$-pointed $G$-set $A$ canonically
gives rise to a crossed homomorphism $\rh_A\colon G\rtarr GL(n,R)$.
\end{cor}

We shall need to embed skew representations in permutation skew representations
to apply these notions in equivariant bundle (or covering space) theory.  Of
course, in classical representation theory over $\bC$, every representation embeds
in a permutation representation. We need an analogue for skew representations.

\begin{defn} A $G$-ring $R$ is amenable if there is a monomorphism of 
$R_G[G]$-modules that embeds any finite dimensional skew representation 
of $G$ over $R$ into a finite dimensional permutation skew representation.
\end{defn} 

\begin{exmp} Let $G$ act trivially on $\mathbf{n}=\{1,\cdots,n\}$. The trivial permutation skew representation $R[\mathbf{n}]$ is the $R_G[G]$-module 
corresponding to the trivial crossed homomorphism $\epz\colon G\rtarr GL(n,R)$.
Thus, when $H^1(G;GL(n,R))= [\epz]$ for all $n$, every skew representation 
of $G$ over $R$ is isomorphic to a permutation skew representation and $R$ is 
amenable. This holds, for example, when $G$ is the Galois group of a Galois 
extension $R=K$ over a field $k$.
\end{exmp}

More generally, we have the following analogue of the situation in classical representation theory.  It shows that amenability is not an unduly restrictive
condition. The following result is in Passman \cite[4.1 in Chapter 1]{Pass}.
Even in this generality, he ascribes it to Maschke. 

\begin{lem} Let $N\subset M$ be $R_G[G]$-modules with no $|G|$-torsion.  If
$M = N\oplus V$ as an $R$-module, then there is an $R_G[G]$-submodule $P\subset M$
such that $|G|M\subset N\oplus P$. 
\end{lem}

An irreducible skew representation is one that has no
non-trivial proper skew subrepresentations.  

\begin{thm} Suppose that $R$ is semisimple and $|G|^{-1}\in R$. Then every
$R_G[G]$-module is completely reducible and $R$ is amenable.
\end{thm}
\begin{proof} By the lemma, if $N\subset M$, then $M=N\oplus P$. That is, the
complete reducibility of $R$-modules implies the complete reducibility of
$R_G[G]$-modules.  If $N$ is an irreducible $R_G[G]$-module, then any choice
of an element $n\neq 0$ determines a map of $R_G[G]$-modules $f\colon R_G[G]\rtarr N$
such that $f(1) = n$. The image of $f$ is a submodule of $N$, and it is all of $N$
since $N$ is irreducible. By complete reducibility, $\Ker(f)$ has a complement in
$R_G[G]$, and that complement must be isomorphic to $N$.  Thus $N$ is a direct 
summand of the permutation skew representation $R_G[G]$. Therefore, by complete reducibility, 
all skew representations are direct summands of permutation skew representations.
\end{proof} 

\subsection{A model $\widetilde{\sG\! \sL}_G(n,R)$ for $E(G,GL(n,R)_G)$}

Again let $R$ be a $G$-ring, and assume that $R$ is amenable.  
We have the entrywise left action of $G$ on $GL(n,R)$, and we have the right action
of $GL(n,R)$ on $GL(n,R)$ given by matrix multiplication.  

\begin{lem} The left action of $G$ and the right action of $GL(n,R)$
on $GL(n,R)$ specify an action of $GL(n,R) \rtimes G$ on $GL(n,R)$ via
$(\ta,g)(x) = (gx)\ta^{-1}$ for $g\in G$, $x\in GL(n,R)$, and $\ta\in GL(n,R)$.  
\end{lem}
\begin{proof} The required relation $g\cdot(x\ta) = (g\cdot x)(g\cdot \ta)$
is immediate from the fact that $g\colon R\rtarr R$ is an automorphism of rings.
\end{proof}

Recall the $G$-set $U$ from \myref{UU}. By \myref{permrep}, $R[U]$ is an 
$R_G[G]$-module with 
\begin{equation}\label{R[A]}
g\cdot (ru) = r^g gu\quad \mbox{for $g\in G$, $r\in R$ and $u\in U$.}
\end{equation}
Similarly, we have the entrywise (equivalently, diagonal) left action of $g$ on 
$R^n$,  $g\cdot (re_i) = r^g e_i$, where we think of $G$ as acting trivially on the 
set $\{e_i\}$. Regard elements $\ta\in GL(n,R)$ as homomorphisms 
$\ta\colon R^n\rtarr R^n$.  That fixes the left action of $GL(n,R)$ on $R^n$ 
given by matrix multiplication, where 
elements of $R^n$ are thought of as row matrices.

\begin{defn}\mylabel{ExamR} We define the chaotic general linear category 
$\widetilde{\sG\! \sL}_G(n,R)$. The objects of $\widetilde{\sG\! \sL}_G(n,R)$ are the
monomorphisms of left $R$-modules $\al\colon R^n\rtarr R[U]$.  Let $G$ act
from the left on objects by $g\al = g\com \al\com g^{-1}$.  By (\ref{R[A]}), 
we have
\begin{eqnarray*}\label{galpha}
(g\com \al\com g^{-1})(\sum r_i e_i) &=&
  \sum_i (g\com \al)  (r_i^{g^{-1}}\, e_i)) = \sum_i g(r_i^{g^{-1}})\al(e_i)\\
&= & \sum_i r_i^{g^{-1}g} (g\cdot \al(e_i))
 =  \sum_i r_i (g\cdot \al(e_i)).\\
\end{eqnarray*}
In particular, $(g\al)(e_i) = g\cdot \al(e_i)$.
Let $GL(n,R)$ act from the right on objects by  
$\al\ta = \al\com \ta\colon R^n\rtarr R[U]$; this uses the left,
not the right, action of $GL(n,R)$ on $R^n$.  Since $\widetilde{\sG\! \sL}_G(n,R)$ is chaotic, 
this fixes the actions on the morphism set, which the map $(S,T)$ identifies with the product 
of two copies of the object set.
\end{defn}

\begin{prop} The actions of $G$ and $GL(n,R)$ on $\widetilde{\sG\! \sL}_G(n,R)$
determine a left action of $GL(n,R) \rtimes G$ via 
\[ (\ta,g)\al  = (g\al)\ta^{-1}. \]
The classifying space 
$|N\widetilde{\sG\! \sL}_G(n,R)|$ is a universal principal $(G,GL(n,R)_{G})$-bundle.
\end{prop}
\begin{proof} For the first claim, we must show that $g(\al\ta) = (g\al)(g\cdot \ta)\colon R^n\rtarr R[U]$ for 
$\al\colon R^n\rtarr R[U]$, $g\in G$, and $\ta = (t_{i,j})\in GL(n,R)$. On elements 
$e_i$, 
\begin{eqnarray*}
g(\al\ta)(e_i) & = & g\cdot (\al\ta)(e_i)
= g\cdot (\al (\sum_j t_{i,j}e_j) \\
&=& g\cdot \sum_j (t_{i,j} \al(e_j)) 
= \sum_j t_{i,j}^g(g\cdot \al(e_j)) \\
&=& (g\al)(\sum_j t_{i,j}^g\, e_j) = (g\al)(g\cdot \ta)(e_i). \\
\end{eqnarray*}
For each free $R$-module $M\subset R[U]$, choose an $R$-linear isomorphism
$\et_M\colon R^n\rtarr M$.  Sending $\al\colon R^n\rtarr M$ to $\et_M^{-1}\com\al$ 
specifies an isomorphism of right $GL(n,R)$-sets from the set of objects 
$\al$ with image $M$ to $GL(n,R)$; the inverse sends $\ta\in GL(n,R)$ to 
$\et_M\com \ta$. Therefore $GL(n,R)$ acts freely on $\widetilde{\sG\! \sL}_G(n,R)$. 
Since $\widetilde{\sG\! \sL}_G(n,R)$ is chaotic, it only remains to show that the set of 
objects of $\widetilde{\sG\! \sL}_G(n,R)^{\LA}$ is non-empty if $\LA\cap GL(n,R)=\{e\}$.  
By \myref{crosssubgp}, $\LA =\{(\rh(h),h)|h\in H\}$, where $H$ is a subgroup of $G$ 
and $\rh\colon H\rtarr GL(n,R)$ is a crossed homomorphism.

By \myref{CrossR}, we may use $\rh$ to endow $R^n$ with a structure of left 
$R_H[H]$-module. By the assumed amenability of $R$, there is a monomorphism of 
left $R_H[H]$-modules $R^n\rtarr R[A]$ for some finite $H$-set $A$.  We can 
embed $A$ in the finite $G$-set $B =G\times_H A$ and then $B$ is isomorphic 
to a sub $G$-set of $U$.  This fixes a monomorphism $\al\colon R^n\rtarr R[U]$ 
of left $R_H[H]$-modules.  Writing $\rh(h) = (s_{i,j})$ and $\rh(h)^{-1} = (t_{i,j})$,
we have
\[ h\al(e_j) = \al(\rh(h)(e_j)) = \al(\sum_k s_{j,k} e_k)
= \sum_k s_{j,k}\al(e_k)\]
and therefore, using the display in \myref{ExamR},
\[ ((h\al)\rh(h)^{-1})(e_i)  =  (h\al)(\sum_j t_{i,j}e_j)  
=  \sum_j t_{i,j} h\cdot \al(e_j)
=  \sum_j\sum_k t_{i,j}s_{j,k} \al(e_k) =  \al(e_i). \qedhere \]
\end{proof}

\begin{defn}\mylabel{ExamR2} Define $\sG\! \sL_G(n,R)$ to be the orbit
$G$-category $\widetilde{\sG\! \sL}_G(n,R)/GL(n,R)$. 
\end{defn}

The classifying space $|N\sG\! \sL_G(n,R)|$ is a model for $B(G,GL(n,R)_G)$.  
Up to isomorphism, the $G$-category $\sG\! \sL_G(n,R)$ admits the following 
explicit description. 

\begin{lem} The objects of $\sG\! \sL_G(n,R)$ are the $n$-dimensional
free $R$-submodules $M$ of $R[U]$.
The morphisms $\al\colon M\rtarr N$ are the isomorphisms of 
$R$-modules. The group $G$ acts by translation on objects, so  
that $gM = \{gm\,|\,m\in M\}$, and by conjugation on morphisms, 
so that $(g\al)(gm) = \al(m)$ for $m\in M$ and $g\in G$.
\end{lem}
\begin{proof}  The objects $\al$ of $\widetilde{\sG\! \sL}_G(n,R)$ with a fixed
image $M$ are all in the same orbit.  Choose $\et_M\colon R^n\rtarr M$ to fix
an orbit representative.  In $\widetilde{\sG\! \sL}_G(n,R)$, we have a unique
morphism $\io\colon \et \rtarr \be$ for each object $\be\colon R^n\rtarr N$. 
We define $\al\colon M\rtarr N$ to be the composite $\be\com \et_M^{-1}$. 
The $\al$ are isomorphisms of $R$-modules that give orbit representatives 
specifying the morphisms of $\sG\! \sL_G(n,R)$. As in the proof of \myref{RightE}, 
the description of the action of $G$ follows.
\end{proof}

\end{document}